\documentclass[12pt]{amsart}
\usepackage{ulem}
\calclayout
\usepackage{float}
\usepackage{amsmath,amsfonts,amssymb,amsrefs}
\usepackage{stmaryrd}
\usepackage{graphicx, epsfig, psfrag}
\usepackage{epstopdf}
\usepackage{color}
\usepackage{enumerate}
\usepackage{scrextend}
\theoremstyle{plain}
\newtheorem{theorem}{Theorem}

\newtheorem{lemma}[theorem]{Lemma}

\newtheorem{remark}[theorem]{Remark}
\newcommand\T{T}
\newcommand\R{\mathbb R}

\newcommand\bq{\boldsymbol q}
\newcommand\br{\boldsymbol r}

\newcommand\s{{\boldsymbol s}}
\newcommand\ba{{\boldsymbol a}}
\newcommand\bJ{{\boldsymbol J}}
\newcommand\bw{{\boldsymbol w}}
\newcommand\bx{{\boldsymbol x}}
\newcommand\bX{{\boldsymbol X}}

\newcommand\bY{{\boldsymbol Y}}

\newcommand\balpha{{\boldsymbol\alpha}}

\newcommand\bnu{{\boldsymbol\nu}}

\newcommand{\D}{\mathcal{D}} 
\newcommand{\tube}{\mathcal{N}} 

\newcommand{\trih}{\mathcal{T}_h}

\newcommand{\intg}{\int_\Gamma}
\newcommand{\intgh}{\int_{\Gamma_h}}
\newcommand{\dD}{\underline{D}}

\newcommand\Ho{\mathaccent23 H^1(\Gamma)}

\def\diag{\operatorname{diag}}
\def\diam{\operatorname{diam}}
\def\dist{\operatorname{dist}}

\def\supp{\operatorname{supp}}
\def\tr{\operatorname{tr}}
\def\lap{\operatorname{\Delta}}
\def\grad{\operatorname{\operatorname{\nabla}}}
\def\bgrad{\operatorname{\boldsymbol{\operatorname{\nabla}}}}
\def\bgradg{\operatorname{\boldsymbol{\nabla}_\Gamma}}
\def\bgradgh{\operatorname{\boldsymbol{\operatorname{\nabla}}_{\Gamma_h}}}
\def\Hessian{\operatorname{\operatorname{\nabla_\Gamma^2}}}

\newcommand\ipd[2]{\partial#1/\partial #2}

\definecolor{bluegreen}{rgb}{0,0.75,0.75}

\begin{document}
\title[Finite elements for the Laplace--Beltrami equation using graded meshes]{Analysis of the finite element method for the Laplace--Beltrami equation on surfaces with regions of high curvature using graded meshes}

\author{Johnny Guzman, Alexandre Madureira, Marcus Sarkis and Shawn Walker}

\begin{abstract}
We derive error estimates for the piecewise linear finite element approximation of the Laplace--Beltrami operator on a bounded, orientable, $C^3$, surface without boundary on general shape regular meshes.  As an application, we consider a problem where the domain is split into two regions: one which has relatively high curvature and one that has low curvature.   Using a graded mesh we prove error estimates that do not depend on the curvature on the high curvature region. Numerical experiments are provided.
\end{abstract}

\date{May 10, 2016}
\maketitle

\section{Introduction}
Since the publication of the seminal paper~\cite{MR976234}, there has been a growing interest in the discretization of surface partial differential equations (PDEs) using finite element methods (FEMs). Such interest is motivated by important applications related to physical and biological phenomena, and also by the potential use of numerical methods to answer theoretical questions in geometry~\cites{MR976234,MR3038698,MR3486164}.

In this paper, we focus on linear finite element methods for the Poisson problem with the Laplace--Beltrami operator on $\Gamma\subset\R^3$, a $C^3$ two-dimensional compact orientable surface without boundary. That is, we consider
\begin{equation*}
-\lap_\Gamma u=f\quad\text{on }\Gamma.
\end{equation*}
In order to motivate the results in our paper, we start by giving a short overview of previous results. A piecewise linear finite element method is proposed and analyzed in~\cites{MR976234,MR3038698}. The basic idea is to consider a piecewise linear approximation of the surface, and pose a finite element method over the discretized surface. Discretizing the surface, of course, creates a \emph{geometric error}, however, the advantage is that for a given discretization a surface parametrization is not necessary. 


In~\cite{MR2485433} a generalization of the piecewise linear FEM is considered, based on higher order polynomials that approximate both the geometry and the PDE; the same paper proposes a variant of the method which employs parametric elements, and the method is posed on the surface,  originating thus no geometric error. Discontinuous Galerkin schemes are considered in~\cites{MR3338674,MR3081490}, and HDG and mixed versions are considered in~\cite{MR3522964}. Adaptive schemes are presented in~\cites{Bonito,MR3447136,MR2285862,MR2970758}. An alternative approach, where a discretization of an outer domain induces the finite element spaces is proposed  in~\cite{MR2551197,MR2570076}. See also~\cites{MR3312662,MR3345245,MR2970758,MR3215065,MR3194806}. In~\cite{MR1868103,MR2608464,MR3043557,MR3471100}, the PDE itself is extended to a neighborhood of the surface before discretization. 

Other problems and methods were considered as well, as a multiscale FEM for PDEs posed on rough surface~\cite{MR3053884}, and stabilized methods~\cites{MR3312662,hanslar,MR3371354,MR3194806}. In~\cite{MR2915563} the finite element exterior calculus framework was considered.
Finally, transient and nonlinear PDEs were also subject of consideration, as reviewed in~\cite{MR3038698}. 

A common ground between all aforementioned papers is that the \emph{a-priori} choices of the surface discretization do not consider how to locally refine the mesh following some optimality criterion. It is however reasonable to expect that some geometrical traits, as the curvatures, have a local influence on the solution, and thus the mesh refinement could account for that locally. Not surprisingly, numerical tests using adaptive schemes confirm that high curvature regions require refined meshes~\cites{MR3447136,Bonito}. This is no different from problems in nonconvex \emph{flat} domains, where corner singularities arise, and meshes are used to tame the singularity at an optimal cost~\cites{MR0502065,MR0502067}.

As far as we can tell, the development and analysis of \emph{a-priori} strategies to deal with high curvature regions have not been an object of investigation, so far.  In this paper we consider a simple setting: We suppose that the domain $\Gamma=\Gamma_1\cup\Gamma_2$, and assume that the maximum of the principal curvatures in $\Gamma_1$ is much larger than those in $\Gamma_2$. We then seek  a graded mesh that gives us optimal error estimate. Of course, in the region $\Gamma_1$ the triangles will be much smaller than the mesh size in regions far from $\Gamma_1$.  We consider the method originally proposed by Dziuk~\cite{MR976234}.

To carry out the analysis, we first need to track the geometric constants carefully. This, as far as we can tell, has not explicitly appeared in literature, although it is not a difficult task. We do this by following~\cites{MR976234,MR3038698} although in some cases we give different arguments while trying to be as precise as possible. The estimate we obtain is found below in~\eqref{apriori}. If $u_h^\ell$ is the finite element solution approximation to $u$ then the result reads (see sections below for precise notation):
\begin{equation*}
  \|\bgradg(u-u^\ell_h)\|_{L^2(\Gamma)}
  \le Cc_p[(\Lambda_h+\Psi_h)\|f\|_{L^2(\Gamma)}
  +\|f-f_h^\ell\|_{L^2(\Gamma)}]
  +C\biggl(\sum_{T\in\T_h}h_T^2\|\Hessian u\|_{L^2(T^\ell)}^2\biggr)^{1/2}.
\end{equation*}
Here $f_h^\ell$ is an approximation to $f$, $c_P$ is the Poincar\'e's constant, and the numbers $\Lambda_h$, $\Psi_h$ are geometric quantities. For instance, $\Psi_h=\max_T\kappa_T^2h_T^2 $ where $h_T$ is the diameter of the triangle $T$ and $\kappa_T$ measures the maximum principal curvatures on $T^\ell$ ($T^\ell$ is the surface triangle corresponding to $T$; see sections below).  The important point here is that  $\Lambda_h+\Psi_h$ can be controlled locally. That is, if one wants to reach a certain tolerance, one needs to make $h_T$ small enough only depending on the geometry in $T^\ell$.

On the other hand, $\|\Hessian u\|_{L^2(T^\ell)}$ does not depend only on the local geometry. In order to deal with this term, in the case of two sub-regions, we prove local $H^2$ regularity results. Combining the local regularity and the a-priori error estimate~\eqref{apriori}  we are able to define a mesh grading strategy and prove Theorem~\ref{t:meshdistrib}. The error estimate contained in Theorem~\ref{t:meshdistrib} is independent of the curvature in region $\Gamma_1$, and in some sense is the best error estimate one can hope for given the available information.

The paper is organized as follows. In Section~\ref{s:prelim} we set the notation and derive several fundamental estimates highlighting the influence of the curvatures. Section~\ref{s:fe} regards the finite element and interpolation approximations. Finally, we present in Section~\ref{s:mesh} a local $H^2$ estimate and a mesh grading scheme culminating in an error estimate that is independent of the ``bad'' curvature. The paper ends with numerical results in Section~\ref{s:numerics}.


\section{Preliminaries}\label{s:prelim}
As mentioned above, we assume that $\Gamma$ is bounded, orientable, $C^3$ surface without boundary.   Furthermore, we assume that there is a high curvature region $\Gamma_1\subsetneq\Gamma$, and define $\Gamma_2=\Gamma\backslash\overline{\Gamma_1}$. For $f\in L^2(\Gamma)$ with $\intg f\,dA=0$, let $u\in\Ho$ be such that
\begin{equation}\label{e:contin_laplace_beltrami}
  \intg\bgradg u\cdot\bgradg v\,dA=\intg fv\,dA\qquad\text{for all }v\in\Ho,
\end{equation}
where $\Ho=\{v\in H^1(\Gamma):\,\intg v\,dA=0\}$. We denote by $\bgradg$ the tangential gradient~\cite{MR3486164}, and~\eqref{e:contin_laplace_beltrami} is nothing but the weak formulation of the Poisson problem for the Laplace--Beltrami operator. Existence and uniqueness of solution follows easily from the Poincar\'e's inequality (Lemma~\ref{l:poincare}) and the Lax--Milgram theorem. Details on
the definition of $\bgradg v$ are given below. 
Consider now a triangulation $\Gamma_h$ of the surface $\Gamma$. By that we mean that $\Gamma_h$ is a two-dimensional compact orientable polyhedral $C^0$ surface, and denoting by $\trih$ the set of closed nonempty triangles such that $\cup_{T\in\trih}T=\Gamma_h$, we assume that all vertices belong to $\Gamma$, and that any two triangles have as intersection either the empty set, a vertex or an edge. Let $h_T=\diam(T)$ and  $h=\max\{h_T:\,T\in\trih\}$. For all $T\in\trih$ assume that there is a three-dimensional neighborhood $\tube_T$ of $T$ where for every $\bx\in\tube_T$ there is a unique closest point $\ba(\bx)\in\Gamma$ (see Figure~\ref{fig:Closest_Point_Diagram}) such that
\begin{equation}\label{e:closest_pt_map}
  \bx=\ba(\bx)+d(\bx)\bnu(\bx),
\end{equation}
$d(\bx)$ is the \emph{signed distance} function from $\bx$ to $\Gamma$ and $\bnu(\ba(\bx))$ is the unit normal to $\Gamma$ at $\ba(\bx)$, that is, $\bnu(\ba(\bx))=(\bgrad d(\bx))^t$; with an abuse of notation, we define $\bnu(\bx)=\bnu(\ba(\bx))$ for $\bx\in\tube_T$. Note that by using \emph{local} tubular neighborhoods, we avoid unnecessary restrictions on the mesh size.

 Here we would like to explain some notational conventions that we use. From now on, the gradient of a scalar function will be a row vector. The normal vector $\bnu$ is a column vector (as well as $\bnu_h$ which is defined below).

\begin{figure}
\begin{center}
\includegraphics[width=4.0in]{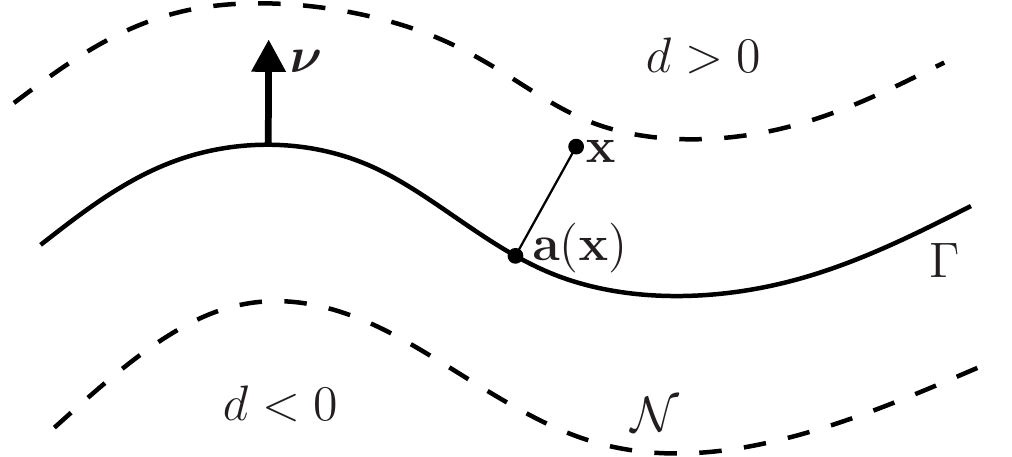}
\caption{Diagram of closest point map~\eqref{e:closest_pt_map}. The exact surface is denoted by $\Gamma$ which is contained in the ``tubular'' neighborhood $\tube$.  Since $d$ is the signed distance function to $\Gamma$, the zero level set of $d$ coincides with $\Gamma$.}
\label{fig:Closest_Point_Diagram}
\end{center}
\end{figure}

We can now define, for every $T\in\trih$, the surface triangle  $T^\ell=\{\ba(\bx):\,\bx\in T\}$. Then $\Gamma=\cup_{T\in\trih}T^\ell$. Let $H(\ba)=\bgradg\bnu(\ba)$ be the self-adjoint operator corresponding to the derivative of the Gauss map~\cite{MR0394451}. Since $\bgradg\bnu=\grad\bnu$, then $H(\ba(\bx))=\grad^2d(\bx)$. And again we set $H(\bx)=H(\ba(\bx))$ for $\bx\in\tube_T$.

At the estimates that follow in this paper we denote by $C$ a generic constant that might not assume the same value at all occurrences, but that does not depend on $h_T$, $u$, $f$ or on $\Gamma$. It might however depend for instance on the shape regularity of $T\in\trih$.

Given $T\in\trih$, let 
\begin{equation}\label{kappa}
\kappa_T=\|H\|_{L^\infty(T)}:=\max_{ij}\|H_{ij}\|_{L^\infty(T)}, 
\end{equation}
and $\bnu_h\in\R^3$ be unit-normal vector to $T$ such that $\bnu_h\cdot\bnu>0$. We note that since $H$ is symmetric, $\kappa_T$ is also equivalent 
to the $L^\infty$ norm of the spectral radius of $H$, or $\max\{|k_1|_{L^\infty(T)},|k_2|_{L^\infty(T)}\}$ where the $k_i$ are the principal curvatures. 

{\bf Assumption}: { \it Throughout the paper we will assume
\begin{equation}\label{suff}
h_T^2 \kappa_T^2 \le c_1 < 1\quad\text{for all }T\in\trih,
\end{equation}
where $c_1$ is sufficiently small.}

It is easy to see that 
\begin{gather}
\|d\|_{L^\infty(T)}\le Ch_T^2\kappa_T,\label{e:d}
\\
\|\bnu-\bnu_h\|_{L^\infty(T)}\le Ch_T\kappa_T.\label{e:nu}
\end{gather}
To show~\eqref{e:d}, assume without loss of generality that $T\subset\R^2\times\{0\}$ and that $\nu_3$, $\nu_{3,h}>0$. Let $I_h d$ be the Lagrange linear interpolant of $d$ in $T$. Since $d$ vanishes at the vertices, $I_hd\equiv0$. By~\cite{MR2050138} we have 
\begin{equation}\label{e:dest}
\|d\|_{L^\infty(T)}+h_T\|\ipd d{x_i}\|_{L^\infty(T)}
=\|d-I_hd\|_{L^\infty(T)}+h_T\|\ipd{(d-I_hd)}{x_i}\|_{L^\infty(T)}
\le Ch_T^2\|\grad^2d\|_{L^\infty(T)},
\end{equation}
for $i=1,2$. Next, to prove~\eqref{e:nu}, start by noting that $\bnu^t_h=(0,0,1)$, and then the estimate for the first two components $\|\nu_i\|_{L^\infty(T)}=\|\ipd d{x_i}\|_{L^\infty(T)}$ of  $\bnu$ follow from~\eqref{e:dest} and \eqref{kappa}. The third component estimate follows from
\begin{multline*}
\|\nu_3-1\|_{L^\infty(T)}\le\|(\nu_3-1)(\nu_3+1)\|_{L^\infty(T)}
=\|\nu_3^2-1\|_{L^\infty(T)}\le\|\nu_1\|_{L^\infty(T)}^2+\|\nu_2\|_{L^\infty(T)}^2
\\
\le Ch_T^2 \kappa_T^2 \le Ch_T \kappa_T. 
\end{multline*}
Here we used~\eqref{suff}.

From~\eqref{e:d} and~\eqref{kappa} we see that
\begin{equation}\label{e:dH}
\|d H\|_{L^\infty(T)} \le C h_T^2 \kappa_T^2.
\end{equation}

Therefore,  making $c_1$  sufficiently small  in \eqref{suff} so that the eigenvalues of  $d(\bx)H(\bx)$ are smaller or equal to $1/2$ for every $\bx \in \Gamma_h$, then we will have
\begin{equation}\label{e:inv}
\|(I-d H)^{-1}\|_{L^\infty(T)}\le C.
\end{equation}

We  define tangential projections onto $\Gamma$ and $\Gamma_h$, respectively, as $P=I-\bnu\otimes\bnu$  and $P_h=I-\bnu_h\otimes\bnu_h$, where $\bq\otimes\br=\bq\br^t$ for two column vectors $\bq$ and $\br$. We recall that the tangential derivatives for a functions defined on a neighborhood of $\Gamma$ (or $\Gamma_h$) are given by
\begin{equation}\label{tanP}
\bgradg v= (\grad v) P, \qquad \bgradgh v=(\grad v)P_h. 
\end{equation}

By using that $\bnu\cdot\bnu=1$, we have
\begin{equation}\label{e:Hnu}
0=\frac12\grad(\bnu\cdot\bnu)=(\grad\bnu)\bnu=H(\bx)\bnu(\bx)\text{ for all }\bx\in T, \end{equation}
Hence, we, of course, have 
\begin{equation}\label{PH}
P H=H=HP,
\end{equation}
which we use repeatedly. Also, we can show that
\begin{equation}\label{e:Hinv}
P(I-d H)^{-1} \bnu=0.
\end{equation}
Indeed, $\bnu=(I-d H)\bnu$ by~\eqref{e:Hnu} and so $P(I-d H)^{-1}\bnu=P(I-d H)^{-1}(I-dH)\bnu=P \bnu=0$.

\subsection{Local parametrization}
Let $\hat T=\{(\theta_1,\theta_2):\,0\le\theta_1,\theta_2\le1,0\le\theta_1+\theta_2\le1\}$ be the reference triangle. Fix $T\in\trih$, let $\bx_0$ be one of the vertices, and let $\bx_1$ and $\bx_2$ be vectors in $\R^3$ representing two edges of $T$ (i.e. $T=\{\bx_0+\theta_1\bx_1+\theta_2\bx_2:\,0\le\theta_1,\theta_2\le1,0\le\theta_1+\theta_2\le1\}$). Let $\bX:\hat T\rightarrow T$ be given by  $\bX(\theta_1,\theta_2)=\bx_0+\theta_1\bx_1+\theta_2\bx_2$. We also define $\bY:\hat T\rightarrow T^\ell$ by $\bY(\theta_1,\theta_2)=\ba(\bX(\theta_1,\theta_2))$. Since $\grad\bX=[\bx_1,\bx_2]$ we have $(\grad\bX)^t\bnu_h=0$. From the definition of $\ba$ we have 
\begin{equation*}
\grad \ba(\bx)= P(\bx)- d(\bx) H(\bx),
\end{equation*}
and, hence
\begin{equation}\label{e:YX}
\grad \bY= (P- d H) \grad X.
\end{equation}
Therefore, using that $P$ and $H$ are symmetric and \eqref{e:Hnu} we have $(\grad \bY)^t \bnu= 0$. Collecting the two results we have
\begin{equation}\label{e:XY}
(\grad \bX)^t\bnu_h=0, \qquad (\nabla \bY)^t \bnu=0.
\end{equation}

Given a function $\eta\in L^1(T^\ell)$ we define the pullback lift $\eta_\ell\in L^1(T)$ as
\begin{equation*}
\eta_{\ell}(\bx)=\eta(\ba(\bx)),
\end{equation*}
and for $\eta\in L^1(T)$ we define the push-forward lift $\eta^\ell\in L^1(T^\ell)$ as
\begin{equation}\label{e:lift}
\eta^{\ell}(\ba(\bx))=\eta(\bx),
\end{equation}
and associate $\hat{\eta}:\hat{T}\rightarrow\R$ defined by
\begin{equation*}
\hat{\eta}(\theta_1, \theta_2)=\eta(\bX(\theta_1, \theta_2))=\eta^\ell(\bY(\theta_1, \theta_2)).
\end{equation*}
Note that $(\eta_\ell)^\ell = \eta$ for $\eta\in L^1(T^\ell)$
and $(\eta^\ell)_\ell = \eta$ for  $\eta \in L^1(T)$.

Consider also the metric tensors
\begin{equation*}
  G_{\bX}(\theta_1,\theta_2)
  =\bigl(\grad \bX(\theta_1,\theta_2)\bigr)^t\grad \bX(\theta_1,\theta_2),
  \qquad
  G_{\bY}(\theta_1, \theta_2)
  =\bigl(\grad \bY(\theta_1, \theta_2)\bigr)^t\grad\bY(\theta_1,\theta_2).
\end{equation*}
From the definition of tangential derivative it is possible to show~\cite{MR3486164}*{Section 4.2.1} (see also (2.2) in~\cite{MR3038698}) that for a function $\eta:\Gamma_h\to\R$,
\begin{equation*}
\bgradgh\eta(\bX)=\bgrad\hat\eta G_\bX^{-1}\grad \bX^t,
\qquad
\bgradg\eta^\ell(\bY)=\bgrad\hat\eta G_\bY^{-1}\grad\bY^t,
\end{equation*}
and multiplying by $\grad\bX$ and $\grad\bY$ we gather that 
\begin{equation}
\bgrad\hat\eta=\bgradgh\eta(\bX)\grad \bX,
\qquad
\bgrad\hat\eta=\bgradg\eta^\ell(\bY)\grad \bY.
\end{equation}
Hence,
\begin{gather}
  \bgradgh\eta(\bX)
  =\bgradg\eta^\ell(\bY)\grad \bY G_\bX^{-1}\grad \bX^t,\label{inq105}
\\
\bgradg\eta^\ell(\bY)
=\bgradgh\eta(\bX)\grad \bX G_\bY^{-1}\grad \bY^t.\label{inq106}
\end{gather}

Note that we can also write
\begin{equation}\label{e:PPh}
  P=\grad\bY G_\bY^{-1}\grad\bY^t, \qquad
  P_h=\grad\bX G_\bX^{-1}\grad\bX^t.
\end{equation}
To see that this is the case, note first from~\eqref{e:XY} that $\grad\bY G_\bY^{-1}\grad\bY^t\bnu=0$. Next, consider for $\epsilon>0$ an arbitrary differentiable curve $\s:(-\epsilon,\epsilon)\to\hat T$ and $\balpha(t)=\bY(\s(t))$. Then
\[
\grad\bY G_\bY^{-1}\grad\bY^t\balpha'
=\grad\bY G_\bY^{-1}\grad\bY^t\grad\bY\s'=\balpha'=P\balpha',
\]
since $\balpha'$ is tangent to $\Gamma$. The same arguments hold for the identity regarding $P_h$.

The following identities have appeared in the literature under different forms; see for example~\cites{MR2485433,MR2285862}. Again, we give a proof for completeness and to show the independence of $C$ with respect to $\Gamma$.
\begin{lemma}\label{l:R_Q_bound}
Let $\eta:\Gamma_h\to\R$ be differentiable, and define its forward lift $\eta^\ell$ as in~\eqref{e:lift}. It then holds that
\begin{equation}\label{inq1}
  \bgradgh\eta=(\bgradg\eta^\ell\circ\ba)\,Q\quad\text{on }\Gamma_h,
\end{equation}
where $Q=(I-dH)P_h$, and
\begin{equation}\label{inq2}
\bgradg\eta^\ell\circ\ba=(\bgradgh\eta) R\quad\text{on }\Gamma_h,
\end{equation}
where $R=\bigl[I-\frac{(\bnu_h-\bnu)\otimes(\bnu-\bnu_h)}{\bnu_h\cdot\bnu}\bigr](I-dH)^{-1}P$. Moreover, there exists a constant $C$ such that
\begin{equation}\label{e:inqRQ}
\|Q\|_{L^\infty(T)}+\|R\|_{L^\infty(T)} \le C.
\end{equation}
\end{lemma}
\begin{proof}
Using~\eqref{inq105} and~\eqref{e:YX} we get
\begin{equation*}
\bgradgh\eta(\bX)=\bgradg\eta^\ell(\bY)[I-d(\bX)H(\bX)]\grad\bX G_\bX^{-1}\grad\bX^t
\end{equation*}
where we used that $\bgradg\eta^\ell \bnu \otimes\bnu=0$. Then~\eqref{inq1} follows from~\eqref{e:PPh}.

To prove~\eqref{inq2}, we use~\eqref{e:YX} and~\eqref{PH} to get $\grad\bY=P[I-d(\bX)H(\bX)]\grad\bX$. Hence, we have
\begin{equation}\label{e:gradY}
(\grad \bY)^t= (\grad\bX)^t[I-d(\bX)H(\bX)]P.
\end{equation}
Using \eqref{PH} we get 
\begin{equation*}
  (I-d H)P(I-dH)^{-1}\biggl(I-\frac{\bnu_h\otimes\bnu}{\bnu_h\cdot\bnu}\biggr)
  =P(I-d H)(I-dH)^{-1}\biggl(I-\frac{\bnu_h\otimes\bnu}{\bnu_h\cdot\bnu}\biggr)
  =P\biggl(I-\frac{\bnu_h\otimes\bnu}{\bnu_h\cdot\bnu}\biggr).
\end{equation*}
However,
\begin{equation*}
P\biggl(I-\frac{\bnu_h\otimes\bnu}{\bnu_h\cdot\bnu}\biggr)
=I-\bnu\otimes\bnu-\frac{\bnu_h\otimes\bnu}{\bnu_h\cdot\bnu}
+\frac{(\bnu\otimes\bnu)\bnu_h\otimes\bnu}{\bnu_h\cdot\bnu}
=I-\frac{\bnu_h\otimes\bnu}{\bnu_h\cdot\bnu}.
\end{equation*}
This gives
 \begin{equation}\label{aux201}
  (I-d H)P(I-dH)^{-1}\biggl(I-\frac{\bnu_h\otimes\bnu}{\bnu_h\cdot\bnu}\biggr)=I-\frac{\bnu_h\otimes\bnu}{\bnu_h\cdot\bnu}
  \end{equation}

 So, from~\eqref{aux201}, \eqref{e:gradY} and~\eqref{e:XY} we have
\begin{equation*}
  (\grad\bY)^t(I-dH)^{-1}\biggl(I-\frac{\bnu_h\otimes\bnu}{\bnu_h\cdot\bnu}\biggr)
  =(\grad \bX)^t\biggl(I-\frac{\bnu_h\otimes\bnu}{\bnu_h\cdot\bnu}\biggr)
  =(\grad \bX)^t.
\end{equation*}
Thus, using~\eqref{inq106} and the above identity we gather that
\begin{alignat*}{1}
\bgradg\eta^\ell(\bY)
=&\bgradgh\eta(\bX)\biggl(I-\frac{\bnu\otimes\bnu_h}{\bnu_h\cdot\bnu}\biggr)
(I-dH)^{-1}\grad\bY G_\bY^{-1}\grad\bY^t
\\
=&\bgradgh\eta(\bX)\biggl(I-\frac{\bnu\otimes\bnu_h}{\bnu_h\cdot\bnu}\biggr)(I-dH)^{-1}P,
\end{alignat*}
from~\eqref{e:PPh}. Clearly we have  $(\bgradgh\eta)~\bnu_h\otimes\bnu_h=0=(\bgradgh\eta)~\bnu_h\otimes\bnu$, and $P(I-dH)^{-1}\bnu\otimes\bnu=0$ follows from~\eqref{e:Hinv}. So we get
\begin{equation*}
\bgradg\eta^\ell(\bY)
=\bgradgh\eta
\biggl[I-\frac{(\bnu_h-\bnu)\otimes(\bnu-\bnu_h)}{\bnu_h\cdot\bnu}\biggr]
(I-dH)^{-1}P.
\end{equation*}
Here we used that $(\bnu_h-\bnu)\otimes(\bnu-\bnu_h)= -\bnu_h \otimes \bnu_h-\bnu \otimes \bnu+ \bnu_h \otimes \bnu+\bnu \otimes \bnu_h$. This proves~\eqref{inq2}. Finally,~\eqref{e:inqRQ} follows from~\eqref{e:nu},~\eqref{e:dH},~\eqref{e:inv} and~\eqref{suff}.
\end{proof}

Next, we write an integration identity.
\begin{lemma}
Let $\eta\in L^1(T)$. Then, if $dA$ is the surface measure in $T^\ell$ and $dA_h$ is the surface measure in $T$ it follows that
\begin{equation}\label{int}
\int_{T^\ell} \eta^\ell\,dA= \int_T\eta\delta_T\,dA_h, 
\end{equation}
where
\begin{equation}\label{e:deltaT}
\delta_T=\sqrt{\det(G_\bY G_\bX^{-1})}. 
\end{equation}
\end{lemma}
\begin{proof}
The result follows from the change of variables formulas~\cite{MR0394451}
\begin{equation*}
  \int_{T^\ell}\eta^\ell\,dA=\int_{\hat T}\hat\eta\sqrt{\det G_\bY}\,d\theta_1\,d\theta_2,
  \qquad
  \int_{\hat T}\hat\eta\sqrt{\det G_\bX^{-1}}\,d\theta_1\,d\theta_2=\int_T\eta\,dA_h.
\end{equation*}
 \end{proof}

Combining~\eqref{int},~\eqref{inq1} and~\eqref{inq2}, we have that
\begin{gather}
  \int_T\bgradgh\eta\cdot\bgradgh\psi\,dA_h
  =\int_{T^\ell}\bgradg\eta^\ell Q^\ell \cdot  \bgradg \psi^\ell Q^\ell
  \frac1{\delta_T^\ell}\,dA, \label{grad_grad_form_map_Q}
  \\
  \int_{T^\ell}\bgradg\eta^\ell\cdot\bgradg\psi^\ell\,dA
  =\int_T\bgradgh\eta R\cdot\bgradgh\psi R\delta_T\,dA_h. \label{grad_grad_form_map_R}
\end{gather}


Next, we prove some bounds for $\delta_T$.
\begin{lemma}\label{l:deltaest}
Assuming that~\eqref{suff} holds and defining $\delta_T$ by~\eqref{e:deltaT} we have that
\begin{gather}
\|\delta_T-1\|_{L^\infty(T)}\le Ch_T^2\kappa_T^2, \label{delta}
\\
\|\frac{1}{\delta_T}-1\|_{L^\infty(T)}\le Ch_T^2\kappa_T^2. \label{invdelta}
\end{gather}
\end{lemma}
\begin{proof}
From~\eqref{e:YX} and~\eqref{e:XY} we have
\begin{multline*}
  G_\bY=\grad\bX^t(I-dH-\bnu\otimes\bnu)\grad\bY
  =\grad\bX^t(I-dH)\grad\bY
  \\
  =\grad\bX^t(I-dH)^2\grad\bX
  -\grad\bX^t (I-dH) \bnu\otimes\bnu\grad\bX.
\end{multline*}
Using~\eqref{e:Hnu} we get
\begin{equation*}
  (\grad\bX)^t(I-dH)\bnu\otimes\bnu\grad\bX=\grad\bX^t\bnu\otimes\bnu\grad\bX.
\end{equation*}
By~\eqref{e:XY},
\[
(\grad\bX)^t\bnu_h\otimes\bnu_h\grad\bX
=(\grad\bX)^t\bnu_h\otimes\bnu\grad\bX
=(\grad\bX)^t\bnu\otimes\bnu_h\grad\bX=0.
\]
Hence,
\begin{equation*}
  (\grad\bX)^t\bnu\otimes\bnu\grad\bX
  =(\grad\bX)^t(\bnu-\bnu_h)\otimes(\bnu-\bnu_h)\grad\bX.
\end{equation*}
Therefore, we get
\begin{equation*}
G_\bY=(\grad\bX)^t[(I-dH)^2-(\bnu-\bnu_h)\otimes(\bnu-\bnu_h)]\grad\bX,
\end{equation*}
or $G_\bY=(\grad\bX)^t(I+B)\grad\bX$ where $B=-2dH+d^2H^2-(\bnu-\bnu_h)\otimes(\bnu-\bnu_h)$. Therefore,
\begin{equation*}
G_\bY G_\bX^{-1}=I+M\quad\text{where }M=\grad\bX^tB\grad\bX G_\bX^{-1}.
\end{equation*}
It is clear that $\|\grad\bX\|_{L^\infty(T)}\le Ch_T$. Also, not difficult to see that  $\|G_\bX^{-1} \|_{L^\infty(T)}\le  C h_T^{-2}$. Moreover, using~\eqref{e:nu} and~\eqref{e:dH} we gather that $\|B\|_{L^\infty(T)}\le Ch_T^2\kappa_T^2$. Hence, $\|M\|_{L^\infty(T)}\le Ch_T^2\kappa_T^2$.
Since $M$ is symmetric, consider the spectral decomposition $M=V\Lambda V^{-1}$, where
$\Lambda=\diag(\lambda_1,\lambda_2)$  and $V$ is orthogonal. Denoting the ith column of $V$ by $v_i$, we have that $\lambda_i=v_i^tM v_i$ and then $\|\lambda_i\|_{L^\infty(T)}\le Ch_T^2\kappa_T^2$ (for $i=1,2$). We also note that
\begin{equation*}
G_\bY G_\bX^{-1}=I+V\Lambda V^{-1}=V(I+\Lambda)V^{-1}.
\end{equation*}
Therefore, we obtain
\begin{equation*}
  \delta_T^2
  =\det(G_\bY G_\bX^{-1})
  =\det(V)\det(I+\Lambda)\det(V^{-1})
  =(1+\lambda_1)(1+\lambda_2),
\end{equation*}
which yields
\begin{equation*}
\|\delta_T^2-1\|_{L^\infty(T)}\le Ch_T^2\kappa_T^2.
\end{equation*}
Using that $\delta_T-1=(\delta_T^2-1)/(\delta_T+1)$, we obtain~\eqref{delta}. The inequality~\eqref{invdelta} follows from the previous inequality and the fact $(\delta_T^{-1}-1)=\delta_T^{-1} (1-\delta_T)$.
\end{proof}

We can now state the following result which follows from Lemmas~\ref{l:R_Q_bound} and~\ref{l:deltaest}, and equations~\eqref{grad_grad_form_map_Q},~\eqref{grad_grad_form_map_R}.
\begin{lemma}\label{l:inqeta}
Assuming the hypotheses of Lemmas~\ref{l:R_Q_bound} and~\ref{l:deltaest}, we have that
\begin{equation*}
  \|\bgradg\eta^\ell\|_{L^2(T^\ell)}
  \le C\|\bgradgh\eta\|_{L^2(T)}
  \le C\|\bgradg \eta^\ell\|_{L^2(T^\ell)}.
\end{equation*}
\end{lemma}
In the following, we use the notation $\dD_iu=(\bgradg u)_i$, and write $\Hessian u$ as the $3\times3$ matrix with entries $\dD_i\dD_ju$ (also denoted by $\dD_{ij}u$).
\begin{lemma}\label{l:H2}
Assuming that~\eqref{suff} holds, we have
\begin{equation*}
  \|\grad_{\Gamma_h}^2\eta\|_{L^2(T)}
  \le C\|\Hessian\eta^\ell\|_{L^2(T^\ell)}
  +C(h_T\kappa_T^2+h_T^2\kappa_T\gamma_T)\|\bgradg\eta^\ell\|_{L^2(T^\ell)},
\end{equation*}
where $\gamma_T=\max_{ij}\|\bgrad H_{ij}\|_{L^\infty(T)}$.
\end{lemma}
\begin{proof}
Let $\bw=\bgradgh\eta$. Using~\eqref{inq1} we see that for $\bx \in T$,
\begin{equation*}
  w_i(\bx)
  =(P_h)_{ik}\bigl(I-d(\bx) H(\bx)\bigr)_{kj}\dD_j\eta^\ell(\ba(\bx)),
\end{equation*}
where we use Einstein summation convention.

Using the product rule and the fact that $P_h$ is constant  we have
\begin{equation*}
\bgradgh w_i(\bx)=\bJ_1(\bx)+\bJ_2 (\bx)+\bJ_3 (\bx),
\end{equation*}
where
\begin{equation*}
\begin{split}
  \bJ_1&=-(P_h)_{ik}H_{kj} \dD_j \eta^\ell \circ \ba \bgradgh d,
  \\
  \bJ_2&= -(P_h)_{ik}d  \dD_j \eta^\ell \circ \ba \bgradgh H_{kj},
  \\
  \bJ_3 &=(P_h) _{ik}(I-d H )_{kj} \bgradgh  (\dD_j \eta^\ell \circ \ba).
\end{split}
\end{equation*}
We start with $\bJ_3$ .  Using \eqref{inq1} we have
\begin{equation*}
\bgradgh  (\dD_j \eta^\ell \circ \ba) (x)= \bgradg \dD_j  \eta^\ell (\ba(x)) Q
\end{equation*}
Hence, using~\eqref{e:inqRQ},~\eqref{int},~\eqref{delta} and~\eqref{suff} we get 
\begin{equation*}
\|\bgradgh(\dD_j\eta^\ell\circ\ba)\|_{L^2(T)}\le C\|\bgradg^2\eta^\ell\|_{L^2(T^\ell)}.
\end{equation*}
If we combine this inequality with~\eqref{e:dH} and~\eqref{suff}, we have
\begin{equation*}
\|\bJ_3\|_{L^2(T))}\le C\|\bgradg^2\eta^\ell\|_{L^2(T^\ell)}.
\end{equation*}
Next, using \eqref{tanP} and the fact $ \bnu_h^t P_h=0$, we obtain
\begin{equation*}
\bgradgh d =  (\grad d) P_h=\bnu^t P_h=(\bnu-\bnu_h)^t P_h.
\end{equation*}
Hence, using~\eqref{e:nu},~\eqref{kappa},~\eqref{int},~\eqref{delta} and~\eqref{suff}  yields
\begin{equation*}
\|\bJ_1\|_{L^2(T)} \le  Ch_T \kappa_T^2  \|\bgradg \eta^\ell \|_{L^2(T^\ell)}.
\end{equation*}
Similarly,~\eqref{e:d},~\eqref{int},~\eqref{delta} and~\eqref{suff} yield 
\begin{equation*}
\|\bJ_2\|_{L^2(T)}\le Ch_T^2\kappa_T\gamma_T\|\bgradg\eta^\ell\|_{L^2(T^\ell)}.
\end{equation*}
Combining the above estimates gives the desired result.
\end{proof}

\section{Finite Element Spaces and local approximations}\label{s:fe}
We introduce the following finite dimensional approximation of \eqref{e:contin_laplace_beltrami}.  The finite element space is given by
\begin{equation*}
  S_h=\{v_h\in C^0(\Gamma_h):\,\intgh v_h\,dA_h=0,\,v_h|_T\text{ is linear for all }T\in\trih\},
  \qquad
  S_h^\ell=\{v_h^\ell:\,v_h\in S_h\}.
\end{equation*}
For $f_h\in L^2(\Gamma_h)$ with $\intgh f_h\,dA_h=0$, let $u_h\in S_h$ such that
\begin{equation}\label{e:FEM}
\intgh\bgradgh u_h\cdot\bgradgh v_h\,dA_h=\intgh f_h v_h\,dA_h\qquad\text{for all }v_h\in S_h.
\end{equation}
Existence and uniqueness of the finite-dimensional problem~\eqref{e:FEM} follows from noting that if $u_h$ is a solution with $f_h=0$, then $u_h$ must be constant with zero average. Thus $u_h=0$.

We will need a Poincar\'e's inequality, as follows~\cite{MR3038698}.
\begin{lemma}\label{l:poincare}
Assuming $\Gamma\subset\R^3$ a $C^3$ two-dimensional compact orientable surface without boundary, there exists a constant $c_p$ such that
\begin{equation}\label{poincareinq}
\|\phi\|_{L^2(\Gamma)}\le c_p\|\bgradg\phi\|_{L^2(\Gamma)}\qquad\text{ for all }\phi\in\Ho.
\end{equation}
\end{lemma}
Then we can state a simple energy estimate.
\begin{lemma}\label{l:energy}
Let $u$ solve~\eqref{e:contin_laplace_beltrami}, then
\begin{equation}\label{energy}
\|\bgradg u\|_{L^2(\Gamma)}\le c_p\|f\|_{L^2(\Gamma)}.
\end{equation}
\end{lemma}
Before proving an \emph{a-priori} estimate for $u^\ell_h-u$ we will need to prove an important lemma that measures the inconsistency in going from $\Gamma$ to $\Gamma_h$. First, we need to develop notation to use in the next proof. Since $\delta_TdA_h=dA$ with respect to a given triangle $T$ and its lifting $T^\ell$, let us define
\begin{equation*}
\delta_h(\bx)=\delta_T,\quad\text{if }\bx\in T.
\end{equation*}

\begin{lemma}\label{inconsistency}
Let $v_h$ and $z_h$ belong to $S_h$. Then the following holds
\begin{gather}
|\int_{\Gamma_h}\bgradgh v_h\cdot\bgradgh z_h\,dA_h
-\int_\Gamma\bgradg v_h^\ell\cdot\bgradg z_h^\ell\,dA|
\le C\Psi_h\|\bgradg v_h^\ell\|_{L^2(\Gamma)}\|\bgradg z_h^\ell \|_{L^2(\Gamma)}  \label{incgrad}
\\
|\int_{\Gamma_h}v_hz_h\,dA_h-\int_\Gamma v_h^\ell z_h^\ell\,dA|
\le C\Psi_h \| v_h^\ell \|_{L^2(\Gamma)} \| z_h^\ell \|_{L^2(\Gamma)}  \label{incfun},
\end{gather}
where $\Psi_h=\max_{T\in\trih}\kappa_T^2h_T^2$.
\end{lemma}
\begin{proof}
Using~\eqref{int} we get
\begin{equation*}
|\int_{\Gamma_h}v_hz_h\,dA_h-\int_\Gamma v_h^\ell z_h^\ell\,dA|
=|\int_{\Gamma} v_h^\ell z_h^\ell\,(\frac1{\delta_h^\ell}-1)\,dA| 
\end{equation*}
Hence, \eqref{incfun} follows from \eqref{invdelta}.

To prove~\eqref{incgrad} we use~\eqref{grad_grad_form_map_Q} to get
\begin{equation*}
  \int_{\Gamma_h}\bgradgh v_h \cdot\bgradgh z_h\,dA_h
  =\int_{\Gamma}(\bgradg v_h^\ell Q^\ell)\cdot(\bgradg z_h^\ell Q^\ell)\frac{1}{\delta_h^\ell}\,dA.
\end{equation*}
Using that $\bgradg(\cdot)P=\bgradg(\cdot)$ we have
\begin{equation*}
  \int_{\Gamma_h}\bgradgh v_h \cdot\bgradgh z_h\,dA_h
  =\int_\Gamma\bgradg v_h^\ell M\cdot\bgradg z_h^\ell \,dA,
\end{equation*}
where
\begin{equation*}
M=\frac{1}{\delta_h^\ell}(PQ^\ell)(PQ^\ell)^t.
\end{equation*}
On the other hand, again using that $\bgradg(\cdot)P=\bgradg(\cdot)$ we get
\begin{equation*}
\int_\Gamma\bgradg v_h^\ell\cdot\bgradg z_h^\ell\,dA
=\int_\Gamma\bgradg v_h^\ell P\cdot\bgradg z_h^\ell\,dA. 
\end{equation*}

Hence, 
\begin{equation}\label{aux301}
 | \int_{\Gamma_h}\bgradgh v_h \cdot\bgradgh z_h\,dA_h-\int_{\Gamma}\bgradg v_h^\ell \cdot \bgradg z_h^\ell \,dA|= |\int_{\Gamma}\bgradg v_h^\ell (M-P) \cdot \bgradg z_h^\ell \,dA|.
\end{equation}
We now proceed to bound $(M-P)$. We first use \eqref{PH}, to get on $\Gamma_h$
\begin{equation*}
  PQ^\ell\circ\ba=PQ=P(I-dH)P_h=(I-d H)PP_h.
\end{equation*}
 Hence, we get
 \begin{equation*}
 M\circ\ba=\frac1{\delta_h}(I-d H)S(I-dH),
 \end{equation*}
 where $S=PP_h P$, and then the triangle inequality yields
 \begin{alignat*}{1}
   \|M-P\|_{L^\infty(T^\ell)}
   =&\|\delta_h^{-1}(I-dH)S(I-dH)-P\|_{L^\infty(T)} \\
   \le& \|\delta_h^{-1}(I-dH)S(I-dH)-S\|_{L^\infty(T)}+\|S-P\|_{L^\infty(T)}.
 \end{alignat*}
 Now using~\eqref{e:dH},~\eqref{suff} and~\eqref{invdelta} and the fact that $S$ is bounded we obtain
 \begin{equation*}
  \|\delta_h^{-1}(I-dH) S (I-dH) -S\|_{L^\infty(T)}  \le C h_T^2\kappa_T^2.
 \end{equation*}
Since $P\bnu_h\otimes\bnu P=P\bnu\otimes\bnu_h P=P\bnu\otimes\bnu P=0$ we have that on $\Gamma_h$,
 \begin{equation*}
S=P(I-\bnu_h\otimes\bnu_h)P=P(I-(\bnu_h-\bnu)\otimes(\bnu_h-\bnu))P.
\end{equation*}
Finally using that $P^2=P$ we have
\begin{equation*}
  S-P=-P(\bnu_h-\bnu)\otimes(\bnu_h-\bnu)P.
\end{equation*}
Therefore, it follows from~\eqref{e:nu} that $\|S-P\|_{L^\infty(T)}\le Ch_T^2\kappa_T^2$. Using the previous inequalities we obtain
\begin{equation*}
 \|M-P\|_{L^\infty(T^\ell)}\le Ch_T^2\kappa_T^2,
\end{equation*}
and hence
\begin{equation*}
 \|M-P\|_{L^\infty(\Gamma)}\le C\Psi_h,
\end{equation*}
and thus~\eqref{incgrad} follows from this inequality and~\eqref{aux301}.
\end{proof}

\begin{theorem}\label{t:quasi}
Let $u\in\Ho$ be the solution of~\eqref{e:contin_laplace_beltrami} and let
$u_h\in S_h$ that solves~\eqref{e:FEM}. Assume that $f^\ell_h$ satisfies $\|f_h^\ell\|_{L^2(\Gamma)}\le C\|f\|_{L^2(\Gamma)}$. Then there exists a constant $C$ such that
\begin{equation*}
\|\bgradg (u-u^\ell_h)\|_{L^2(\Gamma)} \le C \min_{\phi_h \in S_h} \| \bgradg (u-\phi^\ell_h)\|_{L^2(\Gamma)} + Cc_p(\Psi_h  \|f\|_{L^2(\Gamma)}  + \|f-f^\ell_h\|_{L^2(\Gamma)} ),
\end{equation*}
where $\Psi_h=\max_{T\in\trih}\kappa_T^2h_T^2$ is as in Lemma~\ref{inconsistency}.
\end{theorem}
\begin{proof}
For an arbitrary $\phi_h\in S_h$, set $\xi_h=u_h -\phi_h$ and
$\xi^\ell_h=u^\ell_h -\phi^\ell_h$. By Lemma~\ref{l:inqeta} we have
\begin{equation*}
\|\bgradg \xi_h^\ell\|_{L^2(\Gamma)} \le C  \|\bgradgh (u_h-\phi_h)\|_{L^2(\Gamma_h)}.
\end{equation*}
Then, we write for an arbitrary constant $c$
\begin{alignat*}{2}
  \|\bgradgh (u_h-\phi_h)\|_{L^2(\Gamma_h)}^2
  =& \int_{\Gamma_h} \bgradgh u_h \cdot \bgradgh\xi_h dA_h- \int_{\Gamma_h} \bgradgh \phi_h \cdot \bgradgh \xi_h\,dA_h
  \\
  =& \int_{\Gamma_h}f_h\xi_h\,dA_h-\int_{\Gamma_h}\bgradgh\phi_h\cdot\bgradgh\xi_h\,dA_h  \quad && \text{ by } \eqref{e:FEM}
  \\
  =& \int_{\Gamma_h} f_h \xi_h\,dA_h-\int_{\Gamma} f\cdot\xi_h^\ell\,dA
  \\
  &+ \int_\Gamma\bgradg u\cdot\bgradg\xi_h^\ell\,dA - \int_{\Gamma_h} \bgradgh \phi_h \cdot \bgradgh \xi_h\,dA_h, &&  \text{ by } \eqref{e:contin_laplace_beltrami}
  \\
  = &J_1+J_2+J_3+J_4, &&
\end{alignat*}
where after using that $\int_\Gamma f\,dA=0=\int_{\Gamma_h}f_h\,dA_h$,
\begin{alignat*}{1} 
  J_1=&\int_{\Gamma_h} f_h(\xi_h-c)\,dA_h- \int_\Gamma f_h^\ell (\xi_h^\ell-c)\,dA,
  \qquad
  J_2=\int_\Gamma(f_h^\ell-f)(\xi_h^\ell-c)\,dA, 
  \\
  J_3=& \int_\Gamma\bgradg(u-\phi_h^\ell)\cdot\bgradg\xi_h^\ell\,dA,
  \qquad
  J_4=\int_\Gamma\bgradg\phi_h^\ell\cdot\bgradg\xi_h^\ell\,dA
  -\int_{\Gamma_h}\bgradgh\phi_h\cdot\bgradgh\xi_h\,dA_h.
\end{alignat*}
By applying~\eqref{incfun} and the Poincar\'e's inequality~\eqref{poincareinq} we get 
\begin{equation*}
J_1 \le C \, c_p \Psi_h  \|f_h^\ell\|_{L^2(\Gamma)} \| \bgradg \xi_h^\ell \|_{L^2(\Gamma)} ,
\end{equation*}
where we chose $c=\frac1{|\Gamma|}\int_\Gamma\xi_h^\ell\,dA$. Using the Cauchy-Schwarz inequality and the Poincar\'e's inequality~\eqref{poincareinq} we have 
\begin{equation*}
J_2 \le  c_p\,  \|f_h^\ell-f\|_{L^2(\Gamma)}\,  \| \bgradg \xi_h^\ell \|_{L^2(\Gamma)} .
\end{equation*}
Using the  Cauchy-Schwarz inequality we get 
\begin{equation*}
J_3 \le \| \bgradg(u-\phi_h^\ell) \|_{L^2(\Gamma)} \,  \| \bgradg \xi_h^\ell \|_{L^2(\Gamma)} .
\end{equation*}
Using \eqref{incgrad} gives 
\begin{equation*}
J_4 \le C \, \Psi_h \|\bgradg \phi_h^\ell \|_{L^2(\Gamma)}  \| \bgradg \xi_h^\ell \|_{L^2(\Gamma)} 
\end{equation*}
Hence, combining the above results we get
\begin{multline*}
  \|\bgradg \xi_h^\ell\|_{L^2(\Gamma)}
  \le Cc_p(\Psi_h\|f\|_{L^2(\Gamma)}+\|f_h^\ell-f\|_{L^2(\Gamma)})
  \\
  +C(\|\bgradg(u-\phi_h^\ell)\|_{L^2(\Gamma)}+\Psi_h\|\bgradg\phi_h^\ell\|_{L^2(\Gamma)}),
\end{multline*}
where we used that  $\|f_h^\ell\|_{L^2(\Gamma)} \le  \|f\|_{L^2(\Gamma)}$.
If we use the triangle inequality and~\eqref{energy} we obtain 
\begin{equation*}
  \|\bgradg \phi_h^\ell\|_{L^2(\Gamma)}
  \le\|\bgradg (u- \phi_h^\ell) \|_{L^2(\Gamma)}+c_p\|f\|_{L^2(\Gamma)}.
\end{equation*}
The result now follows after taking the minimum over $\phi_h\in S_h$ and use the fact that $\Psi_h\le1$ which follows from~\eqref{suff}.
\end{proof}

Let $I_{h,\ell}$ be the standard Lagrange interpolant in $\Gamma_h$ onto $S_h$ and define $I_h\eta^\ell=(I_{h,\ell}\eta)^\ell\in S_h^\ell$.  We then have the following estimate.
\begin{lemma}\label{l:interpolant}
Let $u$ be the solution of~\eqref{e:contin_laplace_beltrami}. Then,
\begin{equation*}
  \|\bgradg(u-I_h u)\|_{L^2(\Gamma)}
  \le Cc_p(\Lambda_h+\Psi_h)\|f\|_{L^2(\Gamma)}
  +\biggl(\sum_{T\in\trih}h_T^2\|\Hessian u\|_{L^2(T^\ell)}^2\biggr)^{1/2},
\end{equation*}
where $\Lambda_h=\max_{T\in\trih}h_T^3\gamma_T\kappa_T$ and $\gamma_T$ was defined in Lemma~\ref{l:H2}.
\end{lemma}
\begin{proof}
Recall that $u_\ell$ is the pullback lift of $u$. Using Lemma~\ref{l:inqeta}, and approximation properties of the Lagrange interpolant, we have
\begin{equation*}
  \| \bgradg(u-I_h u)\|_{L^2(T^\ell)}
  \le C\|\bgradgh(u_\ell-I_{h,\ell}u_\ell)\|_{L^2(T)}
  \le C h_T \|\grad_{\Gamma_h}^2u_\ell\|_{L^2(T)}.
\end{equation*}
We get from Lemma~\ref{l:H2} that
\begin{equation}\label{e:interp}
  \|\bgradg(u-I_h u)\|_{L^2(T^\ell)}
  \le C(h_T\|\Hessian u\|_{L^2(T^\ell)}
  +(h_T^2\kappa_T^2+h_T^3\kappa_T\gamma_T)\|\bgradg u\|_{L^2(T^\ell)}).
\end{equation}
The result follows easily by summing over $T$ and applying~\eqref{energy}.
\end{proof}

We can combine Theorem~\ref{t:quasi} and Lemma~\ref{l:interpolant} to get the following theorem.
\begin{theorem}\label{t:globalest}
Assume that the hypothesis of Theorem~\ref{t:quasi} holds. Then,
\begin{equation}\label{apriori}
  \|\bgradg(u-u_h^\ell)\|_{L^2(\Gamma)}
  \le Cc_p\bigl((\Lambda_h+\Psi_h)\|f\|_{L^2(\Gamma)}+\|f-f^\ell_h\|_{L^2(\Gamma)}\bigr)
  +C\biggl(\sum_{T\in\trih}h_T^2\|\Hessian u\|_{L^2(T^\ell)}^2\biggr)^{1/2}.
\end{equation}
\end{theorem}

\section{Graded meshes for two subdomains}\label{s:mesh}
In this section we consider the case where there is a high curvature region $\Gamma_1\subsetneq\Gamma$, and $\Gamma_2=\Gamma\backslash\overline{\Gamma}_1$. We let $\kappa^{(i)}=\|H\|_{L^\infty(\Gamma_i)}$ and assume that $\kappa^{(1)}\gg\kappa^{(2)}$. We also let $\kappa=\|H\|_{L^\infty(\Gamma)}=\kappa^{(1)}$. We use our above results from the previous sections and local regularity results in order to grade a mesh so that the error is balanced.

We start by just stating the global regularity result which is found in~\cite{MR3038698}. We note that this result does not fit our graded meshing strategy; instead, we establish Lemma~\ref{l:regul}.

\begin{lemma}\label{l:globalregularity}
Assume that $\Gamma$ is a $C^3$ orientable compact surface without boundary, and that $u\in\Ho$ solves~\eqref{e:contin_laplace_beltrami}. Then $u\in H^2(\Gamma)$, and there exists a constant $C$ that is independent of the curvatures of $\Gamma$ such that
\begin{equation*}
  \|\Hessian u\|_{L^2(\Gamma)}\le(1+Cc_p\kappa)\|f\|_{L^2(\Gamma)}.
\end{equation*}
\end{lemma}
Applying Theorem~\ref{t:globalest} and Lemma~\ref{l:globalregularity} it  easily follows that 
\begin{equation}\label{nogood1}
\|\bgradg(u-u_h^\ell)\|_{L^2(\Gamma)}
\le Cc_p\|f-f^\ell_h\|_{L^2(\Gamma)}+C\bigl(h+c_p(\Lambda_h+\Psi_h+\kappa h)\bigr)\|f\|_{L^2(\Gamma)}.
\end{equation}
Then, requiring $\Lambda_h+\Psi_h\le h(1+\kappa)$ we get
\begin{equation}\label{nogood}
\|\bgradg(u-u_h^\ell)\|_{L^2(\Gamma)}
\le Cc_p\|f-f_h\|_{L^2(\Gamma)}+Ch\bigl(1+c_p(1+\kappa)\bigr)\|f\|_{L^2(\Gamma)}.
\end{equation}

Of course, this is not a very good estimate in the case $\kappa=\kappa^{(1)}\gg\kappa^{(2)}$ since we would have the mesh size far away from $\Gamma_1$ still depending on $\kappa^{(1)}$. Instead, we would like to have the mesh size
to only depend on $\kappa^{(2)}$ a distant order one away from $\Gamma_1$. In order to do this we will need a local regularity result, as follows.
\begin{lemma}[Weighted-$H^2$ regularity]~\label{l:regul}
Let $u\in H^2(\Gamma)$ be the solution of~\eqref{e:contin_laplace_beltrami}, and consider the subset  $\Gamma_2\subseteq\Gamma$. Let $\rho\in W^{1,\infty}(\Gamma)$ be such that $\supp\rho\subseteq\Gamma_2$. Then there is an universal constant $C$ such that
\begin{equation}\label{localregularity}
  \|\rho\Hessian u\|_{L^2(\Gamma_2)}
  \le C\|\rho\|_{W^{1,\infty}(\Gamma)}\bigl(1+c_p(1+\kappa^{(2)})\bigr)\|f\|_{L^2(\Gamma)}.
\end{equation}
\end{lemma}
\begin{proof}
Note that $\dD_i(\rho^2\dD_ju)=2\rho \dD_i\rho \dD_ju+\rho^2\dD_{ij}u$, and
\begin{alignat*}{1}
  \int_{\Gamma_2}\rho^2\dD_{ij}u\dD_{ij}u\,dA
  =&\int_{\Gamma_2}\dD_{ij}u\dD_i(\rho^2\dD_ju)\,dA
  -2\int_{\Gamma_2}\rho \dD_{ij}u\dD_i\rho \dD_ju\,dA
  \\
  \le & \int_{\Gamma_2}\dD_{ij}u\dD_i(\rho^2\dD_ju)\,dA
  +2\|\bgradg\rho\|_{L^\infty(\Gamma)}\|\rho \dD_{ij}u\|_{L^2(\Gamma_2)}\|\dD_ju\|_{L^2(\Gamma_2)}
  \\
  \le& \int_{\Gamma_2}\dD_{ij}u\dD_i(\rho^2\dD_ju)\,dA
  +2\|\bgradg\rho\|_{L^\infty(\Gamma)}^2\|\dD_ju\|_{L^2(\Gamma_2)}^2
  +\frac12\|\rho \dD_{ij}u\|_{L^2(\Gamma_2)}^2.
\end{alignat*}
Then
\[
\int_{\Gamma_2}\rho^2\dD_{ij}u\dD_{ij}u\,dA
\le2\int_{\Gamma_2}\dD_{ij}u\dD_i(\rho^2\dD_ju)\,dA+4|\bgradg\rho|_{L^\infty(\Gamma)}^2\|\dD_ju\|_{L^2(\Gamma_2)}^2.
\]
From Lemma~\ref{l:intbyparts},
\begin{alignat*}{1}
  &\sum_{i,j=1}^3\int_{\Gamma_2}\dD_{ij}u\dD_i(\rho^2\dD_ju)\,dA \\
   & =\sum_{j=1}^3\int_{\Gamma_2}\lap_\Gamma u\dD_j(\rho^2\dD_ju)\,dA -\int_{\Gamma_2}\rho^2\bigl(\tr(H)H-2H^2\bigr)\bgradg u\cdot\bgradg u\,dA
  \\
 &  \le\int_{\Gamma_2}\lap_\Gamma u(2\rho\bgradg\rho\cdot\bgradg u+\rho^2\lap_\Gamma u)\,dA
  +C(\kappa^{(2)})^2\int_{\Gamma_2}\rho^2|\bgradg u|^2\,dA
  \\
  & \le C\bigl(\|\bgradg\rho\|_{L^\infty(\Gamma)}\|\rho\lap_\Gamma u\|_{L^2(\Gamma_2)}\|\bgradg u\|_{L^2(\Gamma_2)}
  +\|\rho\lap_\Gamma u\|_{L^2(\Gamma_2)}^2 \\
  & \qquad +(\kappa^{(2)})^2\|\rho\|_{L^\infty(\Gamma)}^2\|\bgradg u\|_{L^2(\Gamma_2)}^2\bigr).
\end{alignat*}
Thus,
\begin{equation*}
\sum_{i,j=1}^3\int_{\Gamma_2}(\rho \dD_{ij}u)^2\,dA
\le C\bigl(\|\rho\lap_\Gamma u\|_{L^2(\Gamma_2)}^2+(\|\bgradg\rho\|_{L^\infty(\Gamma)}+\kappa^{(2)}\|\rho\|_{L^\infty(\Gamma)})^2\|\bgradg u\|_{L^2(\Gamma_2)}^2\bigr).
\end{equation*}
The result now follows after using the energy estimate~\eqref{energy}.
\end{proof}
Lemma~\ref{l:regul} holds with a generic function $\rho$, but we will apply the result with $\rho(\bx)=\dist(\bx,\Gamma_1)$ which is a 1-Lipschitz function.

\subsection{The graded mesh}\label{s:graded_mesh}
We start by defining $d_1=\bigl(1+c_p(1+\kappa^{(2)})\bigr)/(1+c_p\kappa^{(1)})$. We then define  the region $\D_1=\{\bx\in\Gamma:\dist(\Gamma_1,\bx)\le d_1\}$. We also let $\tilde{\mathcal{T}}_h^1=\{T\in\mathcal{T}_h:\,T^\ell\cap\D_1\ne\emptyset\}$ and define $\tilde{\mathcal{T}}_h^2=\{T\in\trih:\,T\notin\tilde{\mathcal{T}}_h^1\}$.  We set $h_1 =hd_1$. Finally, we let $\rho(\bx)= \dist(\bx, \Gamma_1)$ and also set $\rho_T=\dist(T^\ell, \Gamma_1)$ for all $T \in \mathcal{T}_h$, where $\rho(\cdot)$ is defined using the geodesic distance~\cite{MR1941909} and $\|\nabla\rho\|_{L^\infty(\Gamma_2)}\le1$.

Our graded mesh will then satisfy:
\begin{labeling}{M3\quad}
\item[(M1)]$\Lambda_h+\Psi_h\le h(1+\kappa^{(2)})$
\item[(M2)]$h_T\le h_1$ for every $T\in\tilde{\mathcal{T}}_h^1$
\item[(M3)]$h_T\le \min\{\rho_T,1\} h$ for every $T\in\tilde{\mathcal{T}}_h^2$
\end{labeling}
We recall that $\Psi_h$ and $\Lambda_h$ were respectively defined in Theorem~\ref{t:quasi} and Lemma~\ref{l:interpolant}.

A few comments are in order. First, note that condition (M1) is completely local, and in the case $O(\kappa^{(1)})=O(\kappa^{(2)})$, condition (M1) would be necessary to get an estimate of the form~\eqref{nogood}, as the argument above~\eqref{nogood} shows.

Note that if $O(\kappa^{(1)})=O(\kappa^{(2)})$ then $O(h_1)=O(h)$. The mesh size for triangles that are unit distance from $\Gamma_1$ (i.e. $\rho_T=O(1)$) can be chosen so that $O(h_T)=O(h)$. In the intermediate region a grading giving by (M3) needs to be satisfied. Finally, note that there is a smooth transition for triangles in the border of $\D_1$. Indeed, if $\rho_T=d_1$ then by (M3), $h_T \le d_1 h=h_1 $

We now state and prove our main result.
\begin{theorem}\label{t:meshdistrib}
  Suppose that $u\in\Ho$ solves~\eqref{e:contin_laplace_beltrami} and $u_h\in S_h$ solves~\eqref{e:FEM}. Assume that the mesh satisfies~\rm{(M1-M3)}. Then we have
\begin{equation}\label{e:yesgood}
  \|\bgradg(u-u_h^\ell)\|_{L^2(\Gamma)}
  \le C\bigl(1+c_p(1+\kappa^{(2)})\bigr)h\|f\|_{L^2(\Gamma)}+Cc_p\|f-f^\ell_h\|_{L^2(\Gamma)}.
\end{equation}
\end{theorem}
Before proving the result let us state a few comments. Note that the right hand side of~\eqref{e:yesgood} looks like the right hand side of~\eqref{nogood} with $\kappa^{(2)}$ instead of $\kappa$.  Therefore, with the available information,~\eqref{e:yesgood} is essentially the best result we can hope for. So, we found a graded mesh where one has a fine mesh in the  region where the curvature is high to get the best error estimate.
\begin{proof} (of Theorem~\ref{t:meshdistrib})
By~\eqref{apriori} and our assumption (M1) we have
\begin{equation*}
  \|\bgradg(u-u_h^\ell)\|_{L^2(\Gamma)}
  \le Cc_p\bigl((1+\kappa^{(2)})h\|f\|_{L^2(\Gamma)}+\|f-f^\ell_h\|_{L^2(\Gamma)}\bigr)
  +C\biggl(\sum_{T\in\trih}h_T^2\|\Hessian u\|_{L^2(T^\ell)}^2\biggr)^{1/2}. 
\end{equation*}
Next, we estimate
\[
\sum_{T\in\mathcal{T}_h}h_T^2\|\Hessian u\|_{L^2(T^\ell)}^2
=\sum_{T\in\tilde{\mathcal{T}}_h^1}h_T^2\|\Hessian u\|_{L^2(T^\ell)}^2
+\sum_{T\in\tilde{\mathcal{T}}_h^2}h_T^2\|\Hessian u\|_{L^2(T^\ell)}^2.
\]

By our (M2) we get
\begin{equation*}
\sum_{T\in\tilde{\mathcal{T}}_h^1}h_T^2\|\Hessian u\|_{L^2(T^\ell)}^2
\le Ch_1^2\|\Hessian u\|_{L^2(\Gamma)}^2,
\end{equation*}
and we gather from Lemma~\ref{l:globalregularity} that
\begin{equation*}
\sum_{T\in\tilde{\mathcal{T}}_h^1}h_T^2\|\Hessian u\|_{L^2(T^\ell)}^2
\le Ch_1^2(1+c_p\kappa^{(1)})^2\|f\|_{L^2(\Gamma)}^2
=Ch^2\bigl(1+c_p(1+\kappa^{(2)})\bigr)^2\|f\|_{L^2(\Gamma)}^2.
\end{equation*}

The other term we bound in the following way by using (M3):
\begin{equation*}
 \sum_{T \in \tilde{\mathcal{T}}_h^2} h_T^2 \|\Hessian u\|_{L^2(T^\ell)}^2
\le\sum_{T \in \tilde{\mathcal{T}}_h^2}\frac{h_T^2}{\rho_T^2}\|\rho\Hessian u\|_{L^2(T^\ell)}^2
\le h^2 \|\rho\Hessian u\|_{L^2(\Gamma_2)}^2.
\end{equation*}
Now using~\eqref{localregularity}, we have
\begin{equation*}
\sum_{T \in\tilde{\mathcal{T}}_h^2}h_T^2\|\Hessian u\|_{L^2(T^\ell)}^2
\le Ch^2\bigl(1+c_p(1+\kappa^{(2)})\bigr)^2\|f\|_{L^2(\Gamma)}^2.
\end{equation*}
The result now follows.
\end{proof}

\subsection{An $L^2(\Gamma)$ estimate}\label{ss:L2}
We now derive an error estimate in the $L^2(\Gamma)$ norm, based on the usual duality argument. We note that the conditions (M1-M3) are no longer enough to guarantee a $h^2$ convergence that is independent of $\kappa^{(1)}$. Actually, (M1) is reinforced by imposing that
\begin{labeling}{M3\quad}
\item[(M4)]$\Psi_h\le h^2 (1+ \kappa^{(2)})$
\end{labeling}
\begin{theorem}
Suppose that $u\in\Ho$ solves~\eqref{e:contin_laplace_beltrami} and $u_h\in S_h$ solves~\eqref{e:FEM}. Assume that $f^\ell_h$ satisfies $\|f_h^\ell\|_{L^2(\Gamma)}\le C\|f\|_{L^2(\Gamma)}$, and that the mesh satisfies~\rm{(M1-M4)}. Then we have
\begin{multline*}
  \|u-u_h^\ell\|_{L^2(\Gamma)}
  \le Ch^2\Upsilon\bigl(\Upsilon+c_p+h(1+\kappa^{(2)})(c_p+h\Upsilon)\bigr)\|f\|_{L^2(\Gamma)}
  \\
  +Cc_p(c_p+h\Upsilon)\|f-f^\ell_h\|_{L^2(\Gamma)}.
\end{multline*}
Here $\Upsilon=1+c_p(1+\kappa^{(2)})$. 
\end{theorem}
\begin{proof}
First, let $v\in\Ho$ be the weak solution of 
\begin{equation*}
  -\lap_\Gamma v=u-\tilde{u}_h^\ell\quad\text{on }\Gamma
\end{equation*}
where
\begin{equation*}
 \tilde{u}_h^\ell=u_h^\ell-\frac1{|\Gamma|}\int_\Gamma u_h^\ell\,dA.
\end{equation*}
Then
\begin{alignat*}{2}
\|u-\tilde{u}_h^\ell\|_{L^2(\Gamma)}^2
=& \int_\Gamma\bgradg v\cdot\bgradg(u-\tilde{u}_h^\ell)\,dA \\
=& \int_\Gamma\bgradg(v-I_hv)\cdot\bgradg(u-\tilde{u}_h^\ell)\,dA  \\
&-\int_\Gamma\bgradg I_hv\cdot\bgradg \tilde{u}_h^\ell\,dA
+\int_\Gamma fI_hv\,dA, 
\end{alignat*}
where we used~\eqref{e:contin_laplace_beltrami}. Then, using~\eqref{e:FEM}, the fact that derivatives of constants are zero, and that $\int_\Gamma f\,dA=0=\int_{\Gamma_h}f_h\,dA_h$, we can show that 
\begin{equation*}
\|u-\tilde{u}_h^\ell\|_{L^2(\Gamma)}^2= J_1+J_2+J_3 +J_4,
\end{equation*}
where  
\begin{alignat*}{1}
J_1=&  \int_{\Gamma}  f_h^\ell (I_h v -c) dA-\int_{\Gamma_h} f_h (I_{h} v-c)_{\ell} dA_h \\
J_2=& \int_\Gamma (f-f_h^\ell) (I_h v-c) \,dA \\
J_3=& \int_\Gamma \bgradg (v-I_hv) \cdot \bgradg(u-u_h^\ell)\,dA, \\
J_4=&  \int_{\Gamma_h}\bgradgh ((I_{h}v)_\ell) \cdot \bgradgh u_h\,dA_h
-\int_\Gamma\bgradg I_h v\cdot\bgradg u_h^\ell\,dA. 
\end{alignat*}
Here we choose $c=\frac1{|\Gamma|}\int_\Gamma I_hv\,dA$. Using~\eqref{incfun}, the Poincar\'e's inequality~\eqref{poincareinq}, and that  $\|f_h^\ell\|_{L^2(\Gamma)}\le C\|f\|_{L^2(\Gamma)}$,
\begin{equation*}
J_1\le Cc_p\Psi_h\|f\|_{L^2(\Gamma)}\|\bgradg (I_h v)\|_{L^2(\Gamma)}. 
\end{equation*}
 Using the Poincar\'e's inequality~\eqref{poincareinq} we get 
 \begin{equation*}
J_2\le Cc_p\|f-f_h^\ell\|_{L^2(\Gamma)} \|\bgradg (I_h v)\|_{L^2(\Gamma)}. 
\end{equation*}
Using~\eqref{incgrad} we get 
 \begin{equation*}
J_4\le C\Psi_h\|\bgradg u_h^\ell \|_{L^2(\Gamma)} \|\bgradg (I_h v)\|_{L^2(\Gamma)} . 
\end{equation*}
Using the triangle inequality and the energy estimate~\eqref{energy} we have
 \begin{equation*}
   J_4
   \le C\Psi_h(\|\bgradg(u_h^\ell-u)\|_{L^2(\Gamma)}
   +c_p\|f\|_{L^2(\Gamma)})\|\bgradg(I_h v)\|_{L^2(\Gamma)} .
\end{equation*}
To bound $J_3$ we use the Cauchy-Schwarz inequality
\begin{equation*}
J_3 \le \| \bgradg (u_h^\ell-u) \|_{L^2(\Gamma)}  \, \|\bgradg (I_h v-v)\|_{L^2(\Gamma)} .
\end{equation*}
Using Lemma~\ref{l:interpolant} we get the estimate 
\begin{alignat*}{1}
 \|\bgradg(v-I_hv)\|_{L^2(\Gamma)}^2
\le & C\, c_p^2 \, (\Lambda_h+\Psi_h)^2 \|u-\tilde{u}_h^\ell\|_{L^2(\Gamma)}^2
+\sum_{T\in\trih}h_T^2\|\Hessian v\|_{L^2(T^\ell)}^2.
\end{alignat*}

As we did in the proof of Theorem~\ref{t:meshdistrib} (using (M1-M3)) we can show that 
\begin{equation*}
\sum_{T\in\trih}h_T^2\|\Hessian v\|_{L^2(T^\ell)}^2 \le  C\, h^2 \Upsilon^2  \|u-\tilde{u}_h^\ell\|_{L^2(\Gamma)}^2.
\end{equation*}
Hence, using (M1) we have 
\begin{equation}\label{aux401}
 \|\bgradg(v-I_hv)\|_{L^2(\Gamma)} \le C h\, \Upsilon \, \|u-\tilde{u}_h^\ell\|_{L^2(\Gamma)}.
\end{equation}
Therefore, we get the bound
\begin{equation*}
J_3 \le C h \Upsilon  \| \bgradg (u_h^\ell-u) \|_{L^2(\Gamma)}  \, \|u-\tilde{u}_h^\ell\|_{L^2(\Gamma)}.
\end{equation*}
Using \eqref{aux401}, the triangle inequality and a energy estimate we have 
\begin{equation*}
 \|\bgradg (I_hv)\|_{L^2(\Gamma)} \le C (c_p+h \Upsilon) \, \|u-\tilde{u}_h^\ell\|_{L^2(\Gamma)}.
\end{equation*}
It follows then that 
\begin{alignat*}{1}
J_1+J_2 +J_4 \le&  C (c_p+h \Upsilon) \\
& \quad (c_p\, \Psi_h \|f\|_{L^2(\Gamma)} +c_p \, \|f-f_h^\ell\|_{L^2(\Gamma)}+ \Psi_h \| \bgradg (u_h^\ell-u) \|_{L^2(\Gamma)} )\|u-\tilde{u}_h^\ell\|_{L^2(\Gamma)}.
\end{alignat*}
Therefore, 
\begin{multline*}
  \|u-\tilde{u}_h^\ell\|_{L^2(\Gamma)}
  \le C(c_p+h\Upsilon)(c_p\Psi_h\|f\|_{L^2(\Gamma)}+c_p\|f-f_h^\ell\|_{L^2(\Gamma)})
  \\
  +C\bigl((c_p+h\Upsilon)\Psi_h+h\Upsilon\bigr)\|\bgradg(u_h^\ell-u)\|_{L^2(\Gamma)}.
\end{multline*}
We gather from (M4) and~\eqref{e:yesgood} that 
\begin{multline}\label{aux601}
  \|u-\tilde u_h^\ell\|_{L^2(\Gamma)}
  \le C(c_p+h\Upsilon)h^2\Upsilon\|f\|_{L^2(\Gamma)}+Cc_p(c_p+h\Upsilon)\|f-f_h^\ell\|_{L^2(\Gamma)}  \\
+C\bigl((c_p+h \Upsilon)(1+\kappa^{(2)})h^2+h\Upsilon\bigr)h\Upsilon\|f\|_{L^2(\Gamma)}. 
\end{multline}
The triangle inequality yields 
\begin{equation}\label{aux501}
  \|u-u_h^\ell\|_{L^2(\Gamma)}
  \le\|u-\tilde u_h^\ell\|_{L^2(\Gamma)}+\frac1{|\Gamma|^{1/2}}|\int_{\Gamma} u_h^\ell\,dA|. 
\end{equation}
We now use that $\int_{\Gamma_h}u_h\,dA_h=0$ and~\eqref{incfun} to get
\begin{equation*}
\frac1{|\Gamma|^{1/2}}|\int_\Gamma u_h^\ell\,dA|
=\frac1{|\Gamma|^{1/2}}|\int_\Gamma u_h^\ell\,dA-\int_{\Gamma_h}u_h\,dA_h|
\le C\Psi_h\|u_h^\ell\|_{L^2(\Gamma)}. 
\end{equation*}
Using the triangle inequality,~\eqref{poincareinq} and~\eqref{energy} we gather that 
\begin{equation*}
  \frac1{|\Gamma|^{1/2}}|\int_\Gamma u_h^\ell\,dA|
  \le C(\Psi_h\|u_h^\ell-u\|_{L^2(\Gamma)}+\Psi_hc_p^2\|f\|_{L^2(\Gamma)}).
\end{equation*}
Finally, from~\eqref{aux501},~\eqref{suff}, and (M4) we have 
\begin{equation*}
  \|u-u_h^\ell\|_{L^2(\Gamma)}
  \le C\|u-\tilde u_h^\ell\|_{L^2(\Gamma)}+Cc_ph^2\Upsilon\|f\|_{L^2(\Gamma)}.
\end{equation*}
The result follows from this inequality and~\eqref{aux601}. 
\end{proof}

\begin{remark}
Although we only proved results for domains without boundaries, we anticipate that our analysis will carry over to surfaces with boundary.   In fact, in the next section we will provide numerical experiments for a surface $\Gamma$ with a boundary.
\end{remark}

\section{Numerical Experiments}\label{s:numerics}

We consider a simple example of a surface with a high curvature ``ridge,'' and show our adapted mesh as well as properties of the solution. 

\subsection{Surface Parameterization}

\begin{figure}
\begin{center}

\includegraphics[width=3.0in]{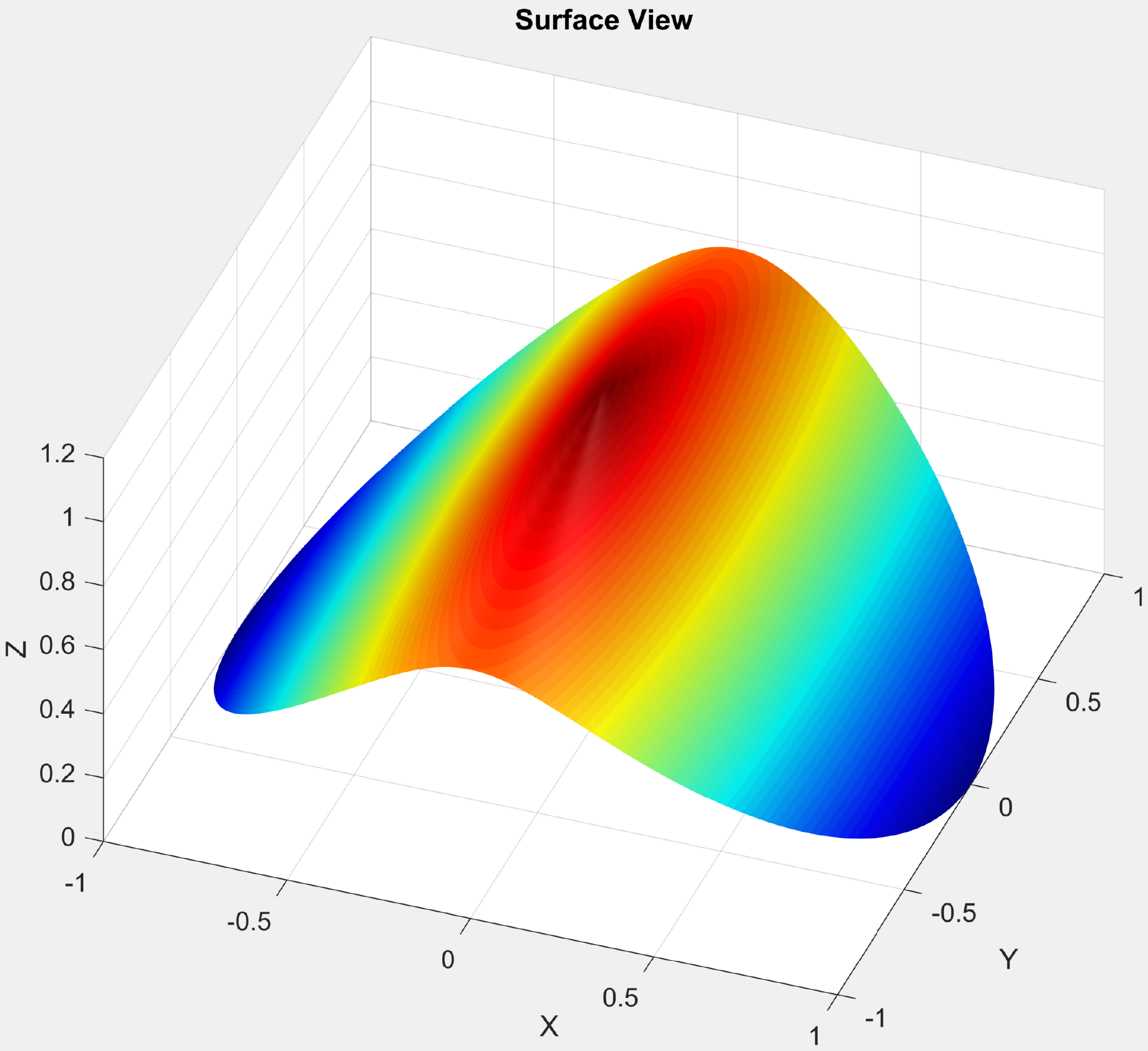}

\caption{Visualization of the surface $\Gamma$.}
\label{fig:Surface_View}
\end{center}
\end{figure}

Let $\Gamma$ be the surface parameterized by
\begin{equation}\label{eqn:surf_param}
  \bX(x,y) = \left(x,y, 1 - \sqrt{x^2 + 5 \cdot 10^{-2} y^2 + 2.5 \cdot 10^{-5} } \right), \quad \text{for } (x,y) \in U,
\end{equation}
where $U = \{ (x,y) \in \R^2 : x^2 + y^2 < 1 \}$ (see Figure \ref{fig:Surface_View}).  The curvature regions $\Gamma_1$, $\Gamma_2$ are defined by
\begin{equation}\label{eqn:domain_partition}
U_1=\biggl\{ (x,y) \in \R^2 : \left(\frac{x}{0.05}\right)^2 + \left( \frac{y}{0.5} \right)^2 \leq 1 \biggr\},\qquad\Gamma_1 = \bX(U_1),
\qquad\Gamma_2 = \Gamma \setminus \Gamma_1.
\end{equation}
This leads to the following maximum curvatures on $\Gamma_1$, $\Gamma_2$:
\begin{equation*}
  \kappa^{(1)} = \max_{\Gamma_1} \kappa = 199.970, \qquad \kappa^{(2)} = \max_{\Gamma_2} \kappa = 8.701, \qquad \frac{\kappa^{(1)}}{\kappa^{(2)}} = 22.984.
\end{equation*}

\subsection{``Exact'' Solution}
We use zero boundary conditions on $\partial \Gamma$ and choose the right-hand-side $f : \Gamma \rightarrow \R$ to be
\begin{equation}\label{eqn:RHS_comp_example}
f(x,y,z)=\begin{cases}
50.0 \exp \left( \frac{1}{ (x-0.2)^2 + y^2 - 0.2 } \right), & \text{if }  (x-0.2)^2 + y^2 < 0.2,
\\
0, & \text{else}.
\end{cases}
\end{equation}
Note that $f$ is similar to a ``bump'' function~\cite{MR1476913} and is $C^\infty(\Gamma)$ and has compact support on $\Gamma$.

In lieu of an exact analytic solution, we compute a reference ``exact'' solution (denoted $u$) on a mesh consisting of 3,679,489 vertices and 7,356,928 triangles obtained from refining an initial coarse mesh.  The number of free degrees-of-freedom of the reference solution is 3,677,441 (after eliminating boundary degrees-of-freedom).  See Figure \ref{fig:Solution} for a plot of $u$.
\begin{figure}
\begin{center}

\includegraphics[width=3.0in]{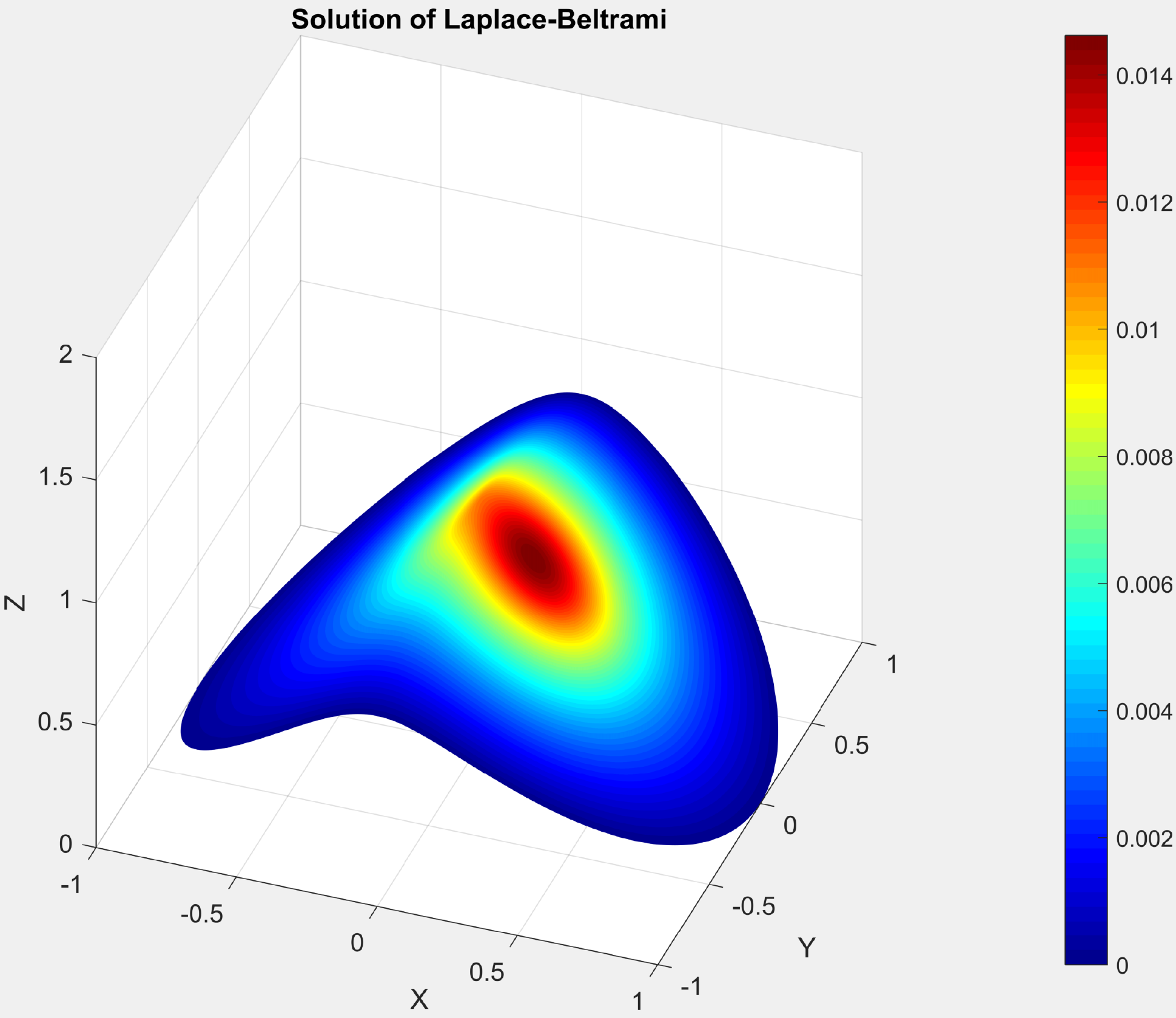}

\caption{Plot of the ``exact'' solution $u$ on $\Gamma$.}
\label{fig:Solution}
\end{center}
\end{figure}

We also plot an approximation of $|\Hessian u|$ to illustrate how the hessian is influenced by the high curvature of the domain, which is concentrated at the high curvature ridge of the surface (see Figure~\ref{fig:surf_hess_soln}).
\begin{figure}
\begin{center}

\includegraphics[width=3.0in]{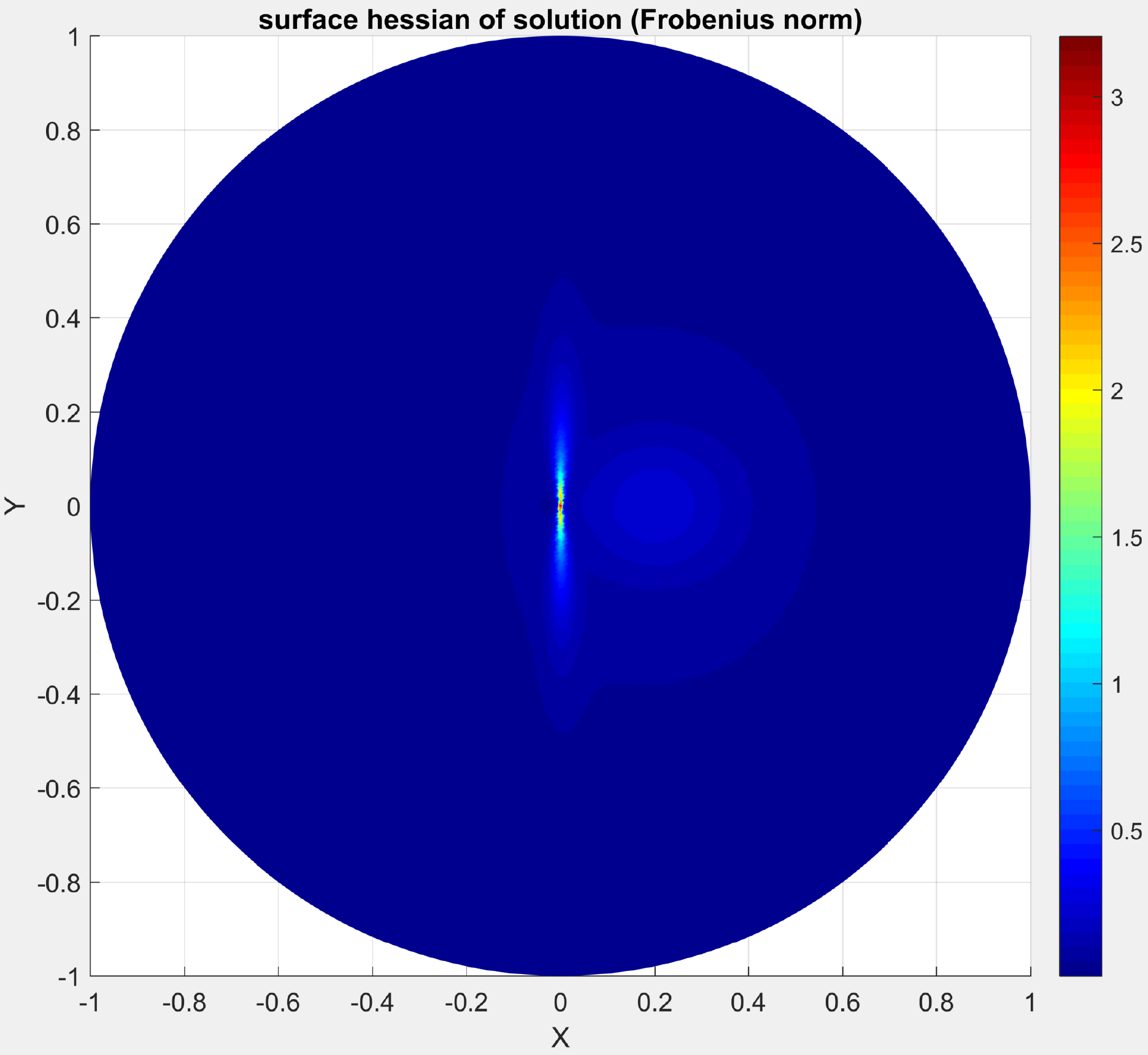}

\caption{Plot of the hessian of $u$ (viewed from the top).  Note that it peaks in the high curvature region.}
\label{fig:surf_hess_soln}
\end{center}
\end{figure}

\subsection{Adapted Mesh And Solution}

Our adapted mesh is generated by first starting with a coarse mesh that satisfies~\eqref{e:dH},~\eqref{suff},~\eqref{e:inv}.  We then iteratively check the criteria in (M1), (M2), (M3) in Section~\ref{s:graded_mesh}.  At each iteration, if any triangle does not satisfy the criteria, then it is marked for refinement.  We then refine all marked triangles using standard longest-edge bisection.  Figure \ref{fig:Adapted_Mesh} shows a plot of our final adapted mesh.
\begin{figure}
\begin{center}

\includegraphics[width=3.0in]{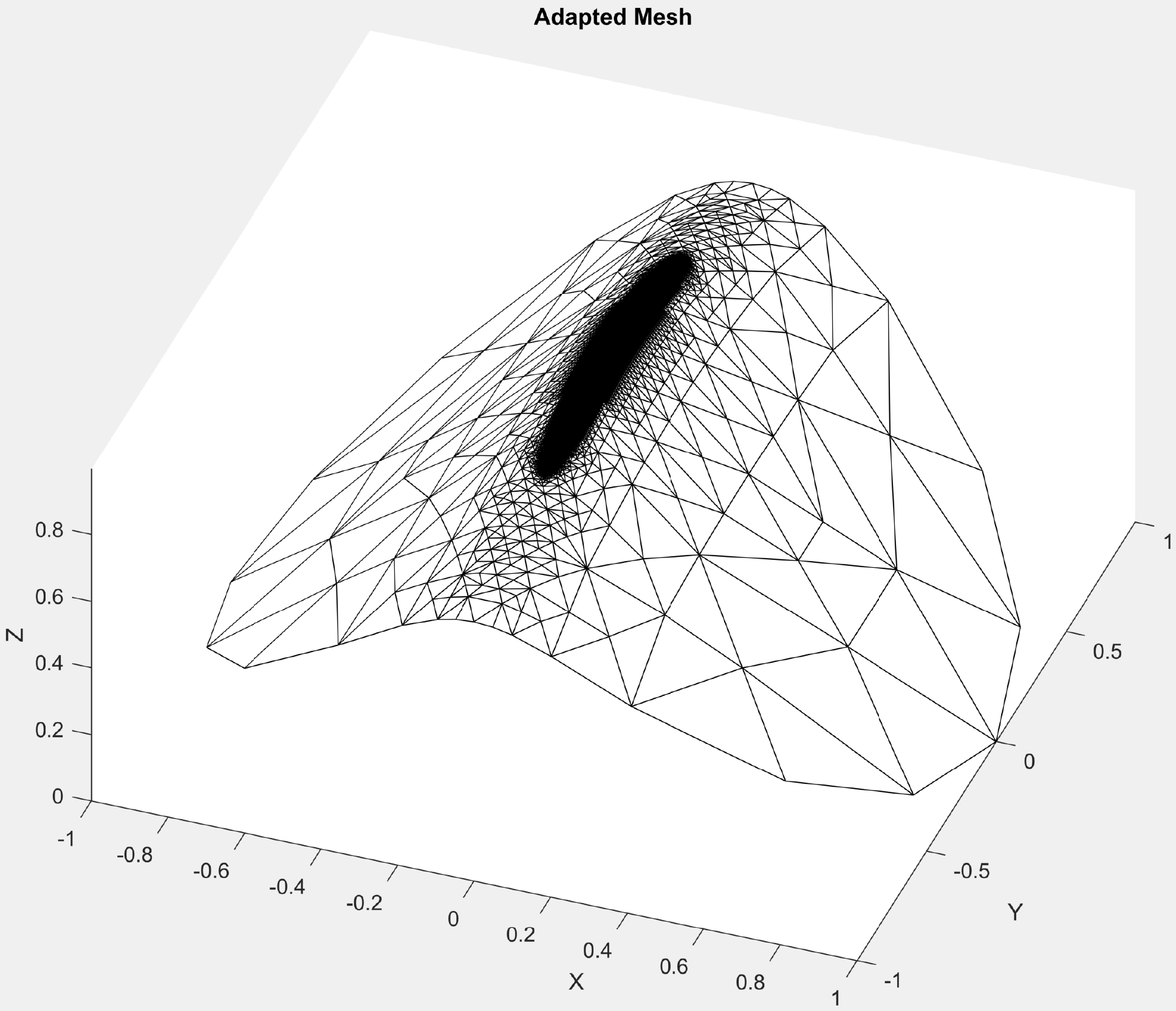}

\caption{Visualization of the adapted mesh of $\Gamma$ using our grading criteria.}
\label{fig:Adapted_Mesh}
\end{center}
\end{figure}

Figure \ref{fig:surf_grad_error} shows the ``pointwise'' error $|\bgradg (u - u_h)|$, where $u_h$ is the numerical solution on the adapted mesh.  Note that the graded mesh, essentially, eliminates the error in the high curvature region.  However, the grading strategy does not specifically account for $f$, so the error is larger where $f$ is large. 
\begin{figure}
\begin{center}

\includegraphics[width=3.0in]{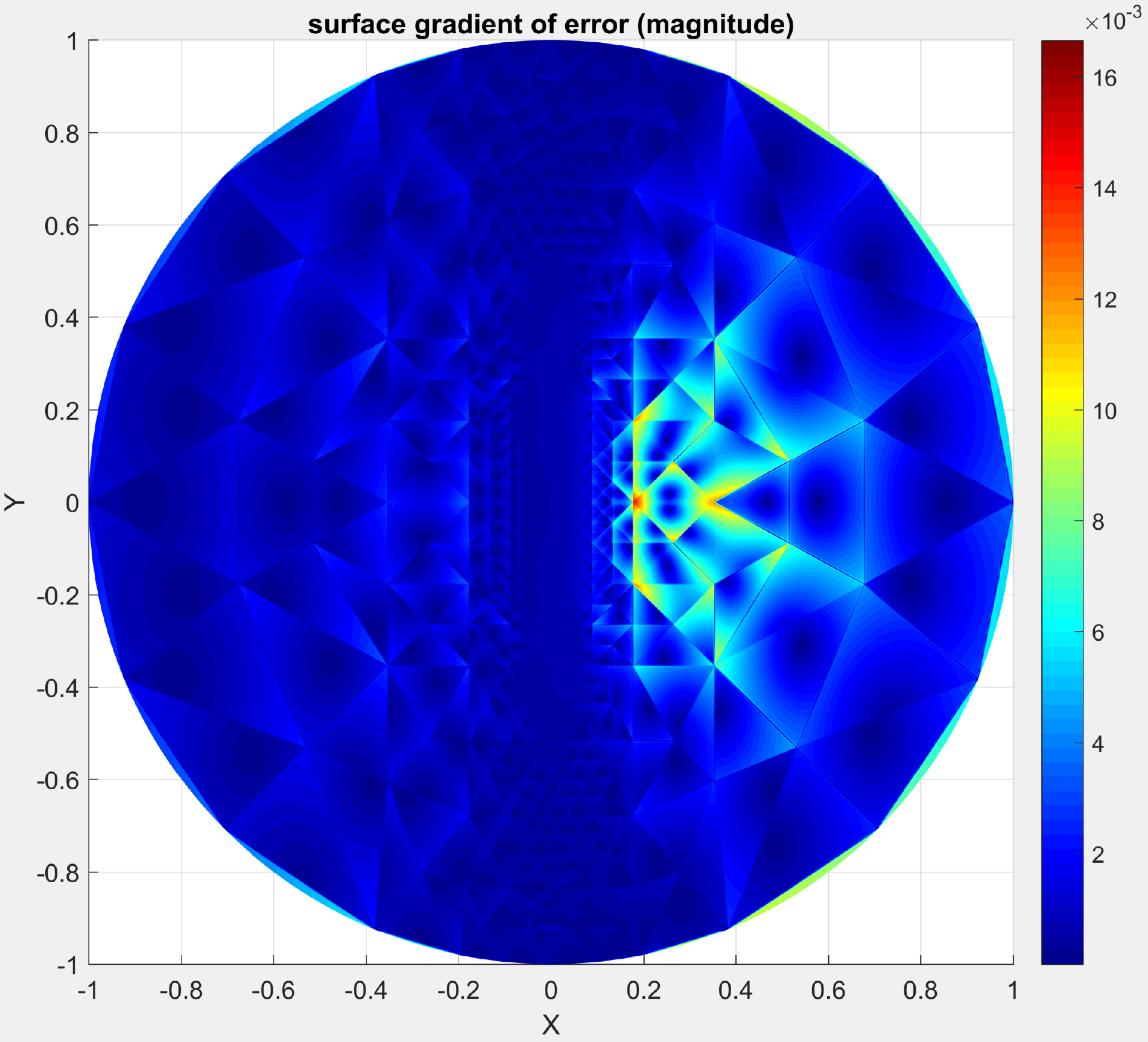}

\caption{Plot of $|\bgradg(u-u_h)|$ (viewed from the top).  The adapted mesh essentially eliminates the error in the high curvature region.}
\label{fig:surf_grad_error}
\end{center}
\end{figure}

\appendix
  \section{One technical result}

  \begin{lemma}\label{l:intbyparts}
Assume that $\Gamma$ is a $C^3$ two-dimensional compact orientable surface without boundary, and that $u\in H^2(\Gamma)$. Then 
\begin{multline*}
\sum_{i,j=1}^3\int_{\Gamma_2}\dD_{ij}u\dD_i(\rho^2\dD_ju)\,dA
\\
=\sum_{j=1}^3\int_{\Gamma_2}\lap_\Gamma u\dD_j(\rho^2\dD_ju)\,dA
-\int_{\Gamma_2}\rho^2\bigl(\tr(H)H-2H^2\bigr)\bgradg u\cdot\bgradg u\,dA. 
\end{multline*}
\end{lemma}
\begin{proof}
In what follows, we use the two identities~\cite{MR3038698}
\begin{equation}\label{app1}
\int_\Gamma\dD_iuv\,dA=-\int_\Gamma u\dD_iv\,dA+\int_\Gamma uv\tr(H)\nu_i\,dA,
\end{equation}
and
\begin{equation}\label{app2}
\dD_{ij}u=\dD_{ji}u+(H\bgradg u)_j\nu_i-(H\bgradg u)_i\nu_j,
 \end{equation}
 for all $C^3(\Gamma)$ functions $u$, $v$. 
 Also we will use that of course
 \begin{equation}\label{app3}
 \bgradg u\cdot\bnu=0.
 \end{equation}

We assume for the proof that $u\in C^3(\Gamma)$, and the general result follows from density arguments. Following~\cite{MR3038698}*{Lemma 3.2}, and using the Einstein summation convention, 
\begin{alignat*}{2}
\int_\Gamma\dD_{ij}u\dD_i(\rho^2\dD_ju)\,dA
=&-\int_\Gamma\dD_{iij}u(\rho^2\dD_ju)\,dA
+\int_\Gamma\dD_{ij}u\rho^2\dD_ju\tr(H)\nu_i\,dA  && \text{ by } \eqref{app1} \\
=&-\int_\Gamma\dD_{iij}u (\rho^2\dD_ju)\,dA\quad && \text{ by } \eqref{app3} \\ 
=&-\int_\Gamma\dD_i\bigl[\dD_{ji}u+(H\bgradg u)_j\nu_i-(H\bgradg u)_i\nu_j\bigr]\rho^2\dD_ju\,dA  \quad && \text{ by } \eqref{app2} 
\\
=&-\int_\Gamma\dD_i\bigl[\dD_{ji}u+H_{jk}\dD_ku\nu_i-H_{ik}\dD_ku\nu_j\bigr]\rho^2\dD_ju\,dA
\\
=&-\int_\Gamma\bigl(\dD_{iji}u+H_{jk}H_{ii}\dD_ku-H_{ik}H_{ij}\dD_ku\bigr)\rho^2\dD_ju\,dA && \text{ by } \eqref{app3}  \\
=&-\int_\Gamma\bigl(\dD_{iji}u \rho^2\dD_ju + \rho^2(\tr(H) H- H^2) \bgradg u \cdot \bgradg u \bigr)\,dA.
\end{alignat*}

To handle the first term on the right hand side, we use \eqref{app2} and the fact that $\dD_{ii} u= \lap_{\Gamma} u$ to write:
\begin{alignat*}{1}
-\int_\Gamma\dD_{iji}u\rho^2\dD_ju\,dA
=&-\int_\Gamma\bigl[\dD_j\lap_\Gamma u+(H\bgradg\dD_iu)_j\nu_i-(H\bgradg\dD_iu)_i\nu_j\bigr]\rho^2\dD_ju\,dA
\\
=&-\int_\Gamma\dD_j\lap_\Gamma u\rho^2\dD_ju+H_{jk}\dD_{ki}u\nu_i\rho^2\dD_ju\,dA,
\end{alignat*}
where we used \eqref{app3} in the last equation.
But $\dD_{ki}u\nu_i=\dD_k(\dD_iu\nu_i)-\dD_iuH_{ki}=-\dD_iuH_{ki}$, and then 
\begin{alignat*}{1}
-\int_\Gamma \dD_{iji} u\rho^2 \dD_ju\,dA
=&\int_\Gamma-\dD_j\lap_\Gamma u\rho^2\dD_ju+H_{jk}\dD_iuH_{ki}\rho^2\dD_ju\,dA
\\
=&\int_\Gamma \lap_\Gamma u\dD_j(\rho^2\dD_ju)+\rho^2H^2\bgradg u\cdot\bgradg u\,dA. 
\end{alignat*}
In the last equation we used~\eqref{app1} and~\eqref{app3}. This completes the proof.
\end{proof}
  

\end{document}